\documentclass[USenglish]{article}	

\usepackage[utf8]{inputenc}				
\usepackage[big,online]{dgruyter}
\usepackage{lmodern} 
\usepackage{microtype}

\usepackage{amsthm,amsfonts,amsopn,braket}
\usepackage{amssymb}
\usepackage{stmaryrd}
\usepackage{booktabs}
\usepackage{scalerel}
\usepackage{graphicx}
\usepackage{comment}
\usepackage{tikz}
\usepackage{pgfplots}
\usepackage{comment}
\usepackage{cleveref}
\usepackage{mathtools}
\usepackage{todonotes}
\pgfplotsset{compat=1.14}

\theoremstyle{dgthm}
\theoremstyle{dgdef}
\newtheorem{theorem}{Theorem}
\newtheorem{corollary}{Corollary}

\newtheorem{lemma}{Lemma}

\newtheorem{remark}{Remark}

\newcommand*{\llbrace}{\{\mskip-6mu\{}
\newcommand*{\rrbrace}{\}\mskip-6mu\}}

\newcommand{\dz}{\partial_z} 

\newcommand{\jump}[1]{\left\llbracket #1 \right\rrbracket}
\newcommand{\avg}[1]{\left\{\!\!\!\left\{ #1 \right\}\!\!\!\right\}}

\newcommand{\flux}[1]{\frac{\mu}{\sigma_t} \dz #1}
\newcommand{\fluxh}[2][]{\frac{\mu}{\sigma_t^{#1}} \dz^h {#2}_{#1}}
\newcommand{\fluxhs}[2][]{\frac{\mu}{\sqrt{\sigma_t^{#1}}} \dz^h {#2}_{#1}}

\newcommand{\Fv}{\mathcal{F}_h^v} 
\newcommand{\boundary}[2]{\ensuremath{\langle {#1},{#2}\rangle}} 

\newcommand{\T}{\mathcal{T}}

\newcommand{\had}{h^{\prime}}

\newcommand{\aad}{a_{\had}}
\newcommand{\uad}{u_{\T^{\p}, \kad}}
\newcommand{\Vad}{V_{\T^{\p}, k}}
\newcommand{\kad}{k^{\prime}}
\newcommand{\norm}[1]{\left\Vert #1 \right\Vert}
\newcommand{\normc}{\left\Vert \cdot \right\Vert}

\newcommand{\p}{\prime}

\newcommand{\ip}[2]{\left( #1, #2\right)}

\newcommand{\kz}{\ensuremath{k_z}}
\newcommand{\kmu}{\ensuremath{k_\mu}}

\begin{document}

\title{A phase-space discontinuous Galerkin scheme for the radiative transfer equation in slab geometry
}
\runningtitle{Phase-space DG for RTE}

\abstract{
    We derive and analyze a symmetric interior penalty discontinuous Galerkin scheme for the approximation of the second-order form of the radiative transfer equation in slab geometry. Using appropriate trace lemmas, the analysis can be carried out as for more standard elliptic problems.
	Supporting examples show the accuracy and stability of the method also numerically, for different polynomial degrees.
	For discretization, we employ quad-tree grids, which allow for local refinement in phase-space, and we show exemplary that adaptive methods can efficiently approximate discontinuous solutions. We investigate the behavior of hierarchical error estimators and error estimators based on local averaging.
}

\author[1]{Riccardo Bardin}
\author[1]{Fleurianne Bertrand}
\author[1]{Olena Palii} 
\author[1]{Matthias Schlottbom} 
\runningauthor{Bardin, Bertrand, Palii and Schlottbom}
\affil[1]{\protect\raggedright 
Faculty of Electrical Engineering, Mathematics and Computer Science\\
University of Twente\\
 Enschede, The Netherlands\\
email: m.schlottbom@utwente.nl}


\maketitle
\section{Introduction}
We consider the numerical solution of the radiative transfer equation in slab geometry, which has several applications such as atmospheric science \cite{Hansen_1974}, oceanography \cite{Arnush_1972}, pharmaceutical powders \cite{Burger:97} or solid state lighting \cite{Melikov_2018}; see also \cite{Carminati_2021} for a recent introduction.

The radiative transfer equation in slab geometry describes the equilibrium distribution of specific intensity $\phi$ in a three-dimensional background medium $\mathbb{R}^2\times (0,L)$ with coordinates $(x,y,z)$ and $L>0$ denoting the thickness of the slab.
The modelled physical principles are propagation, absorption and scattering by the background medium.
The basic assumptions that allow to reduce model complexity are that the scattering and absorption cross sections $\sigma_s$ and $\sigma_a$ are functions of $z$ only, see, e.g., \cite[p. 9]{Agoshkov98}. Moreover, it is assumed that internal sources $\tilde f$ depend only on $z$ and on $\mu\coloneqq s\cdot n_z$, with unit vectors $s\in\mathbb{S}^2$ and $n_z=(0,0,1)^T$. As a consequence, see, e.g,. \cite[p. 9]{Agoshkov98}, the specific intensity $\phi$ is a function of $z$ and $\mu$ only. Assuming, that the distribution of a new direction after a scattering event is distributed uniformly and does not depend on the pre-scattered direction, the stationary radiative transfer equation for the specific intensity with inflow boundary conditions is given by \cite[(1.12)]{Agoshkov98}
\begin{align}\label{eq:rte1}
 \mu \dz \phi(z,\mu)+\sigma_t(z)\phi(z,\mu)&=\frac{\sigma_s(z)}{2} \int_{-1}^1 \phi(z,\mu')\,d\mu' + \tilde f(z,\mu)\quad\text{for } 0<z<L,\ -1<\mu<1,\\
 \phi(0,\mu) &= \tilde g_0(\mu)\quad\text{for } \mu>0, \label{eq:rte2}\\
 \phi(L,\mu) &= \tilde g_L(\mu)\quad\text{for } \mu<0.\label{eq:rte3}
\end{align}
Here, $\sigma_t\coloneqq \sigma_s+\sigma_a$ is called the total cross section and $1/\sigma_t$ describes the mean free path between interactions with the background medium. Moreover, $\tilde g_0$ and $\tilde g_L$ model boundary sources. 
Writing $\phi=\phi^+ + \phi^-$ as a sum of even and odd functions in $\mu$, which are defined by $\phi^\pm(z,\mu)\coloneqq (\phi(z,\mu)\pm \phi(z,-\mu))/2$, a projection of \cref{eq:rte1} onto even and odd functions yields the system, see, e.g., \cite{EggerSchlottbom12},
\begin{align}
 \mu \dz \phi^-(z,\mu)+\sigma_t(z)\phi^+(z,\mu)&=\sigma_s(z)\int_{0}^1 \phi^+(z,\mu')\,d\mu' +\tilde f^+(z,\mu),
 \label{eq:even}\\
 \mu \dz \phi^+(z,\mu)+\sigma_t(z)\phi^-(z,\mu)&=\tilde f^-(z,\mu). 
 \label{eq:odd}
\end{align}
Assuming a strictly positive total cross section $\sigma_t>0$, which is a common assumption in the mentioned applications, we can rewrite \cref{eq:odd} to
\begin{align}\label{eq:odd2}
\phi^-(z,\mu) = \frac{1}{\sigma_t} \left(\tilde f^-(z,\mu)- \mu\dz \phi^+(z,\mu)\right).
\end{align}
Using \cref{eq:odd2} in \cref{eq:even} and in \cref{eq:rte2}--\cref{eq:rte3}, and writing $u(z,\mu)\coloneqq \phi^+(z,\mu)$ for the even part, we obtain the following equivalent second-order form of the radiative transfer equation \cite[(3.76)]{Agoshkov98}, see also \cite{BalMaday2002,EggerSchlottbom12,palii2020convergent},
\begin{alignat}{5}
    -\dz \left(\frac{\mu^2}{\sigma_t}\dz u\right) + \sigma_t u &=\sigma_s \int_0^1 u(\cdot,\mu')\,d\mu' + f &\quad\text{in } \Omega,\label{eq:ep1}\\
     u + \frac{\mu}{\sigma_t} \partial_n u &= g &\quad\text{on }\Gamma.\label{eq:ep2}
\end{alignat}
Here, $\Omega\coloneqq(0,L)\times(0,1)$ and $g(0,\mu)\coloneqq\tilde g(\mu)-\sigma_t^{-1}(0)\tilde f^-(0,\mu)$ and $g(L,\mu)\coloneqq\tilde g_L(\mu)+ \sigma_t^{-1}(L)\tilde f^-(L,\mu)$ for $\mu>0$.
Moreover, $f(z,\mu)\coloneqq \tilde f^+(z,\mu)-\sigma_t^{-1}(z)\mu\dz \tilde f(z,\mu)$.
Furthermore, $\partial_n u(0,\mu)\coloneqq -\dz u(0,\mu)$ and $\partial_n u(L,\mu)\coloneqq\dz u(L,\mu)$ are the normal derivatives of $u$ on the boundary of the slab, defined as $\Gamma\coloneqq \Gamma_0\cup \Gamma_L$, where $\Gamma_z\coloneqq \{z\}\times(0,1)$.
Once $u$ has been determined, the odd part of the specific intensity can be recovered from \cref{eq:odd2}.

Due to the product structure of $\Omega$, it seems natural to use separate discretization techniques for the spatial variable $z$ and the angular variable $\mu$. This is for instance done in the spherical harmonics method, in which a truncated Legendre polynomial expansion is employed to discretize $\mu$ \cite{DuderstadtMartin79}. The resulting coupled system of Legendre moments, which still depend on $z$, is then discretized for instance by finite differences or finite elements \cite{DuderstadtMartin79}.
Another class of approximations consists of discrete ordinates methods which perform a collocation in $\mu$ and the integral in \cref{eq:ep1} is approximated by a quadrature rule \cite{DuderstadtMartin79}. The resulting system of transport equations is then discretized for instance by finite differences \cite{DuderstadtMartin79} or discontinuous Galerkin methods \cite{Han2010,guermond2014discontinuous}, and also spatially adaptive schemes have been used \cite{ragusa2010two}.

A major drawback of the independent discretization of the two variables $z$ and $\mu$ is that a local refinement in phase-space is not possible.
Such local refinement is generally necessary to achieve optimal schemes.
For instance, the solution can be non-smooth in the two points $(z,\mu)=(0,0)$ and $(z,\mu)=(L,0)$, which are exactly the two points separating the inflow from the outflow boundary. Although certain tensor-product grids can resolve this geometric singularity for the slab geometry, such as double Legendre expansions \cite{DuderstadtMartin79}, they fail to do so for generic multi-dimensional situations.
Moreover, local singularities of the solution due to singularities of the optical parameters or the source terms can in general not be resolved with optimal complexity. 

Phase-space discretizations have been used successfully for radiative transfer in several applications, see, e.g., \cite{deAlmeida2017iterative,liu2005finite,Martin_Duderstadt_1977,martin1981phase} for slab geometry, \cite{kitzmann2016discontinuous} for geometries with spherical symmetries, or \cite{favennec_mathew_badri2019ad,kophazi2015space} for more general geometries. Let us also refer to \cite{kitzler2015high} for a phase-space discontinuous Galerkin method for the nonlinear Boltzmann equation.
A non-tensor product discretization that combines ideas of discrete ordinates to discretize the angular variable with a discontinuous Petrov-Galerkin method to discretize the spatial variable has been developed in \cite{dahmen2020adaptive}.
 
In this work, we aim to develop a numerical method for \eqref{eq:ep1}--\eqref{eq:ep2} that allows for local mesh refinement in phase-space and that allows for a relatively simple analysis and implementation. To accomplish this, we base our discretization on a partition of $\Omega$ such that each element in that partition is the Cartesian product of two intervals. Local approximations are then constructed from products of polynomials defined on the respective intervals.
In order to easily handle hanging nodes, which such partitions generally contain, we use globally discontinuous approximations. 
In case the resulting linear systems are very large, iterative solution techniques with small additional memory requirements may be employed for their numerical solution, such as the conjugate gradient method, which, however, requires the linear system to be symmetric positive definite.
Therefore, we employ a symmetric interior penalty discontinuous Galerkin formulation.
Besides the proper treatment of traces, which requires the inclusion of a weight function in our case, the analysis of the overall scheme is along the standard steps for the analysis of discontinuous Galerkin methods \cite{diPietroErn}. As a result, we obtain a scheme that enjoys an abstract quasi-best approximation property in a mesh-dependent energy norm.
Our choice of meshes also allows to explicitly estimate the constants in auxiliary tools, such as inverse estimates and discrete trace inequalities. 
As a result, we can give an explicit lower bound on the penalty parameter required for discrete stability. 
This lower bound for the penalty parameter depends only on the polynomial degree for the approximation in the $z$-variable and is relatively simple to compute; see \cite{EPSHTEYN2007} for the estimation of the penalty parameter in the context of standard elliptic problems.
Our theoretical results about accuracy and stability of the method are confirmed by numerical examples, which show optimal convergence rates for different polynomial degrees assuming sufficient regularity of the solution. Moreover, we show that adaptively refined grids are able to efficiently construct approximations to non-smooth solutions.

For the local adaptation of the grid we investigate several error estimators. First, we consider two hierarchical error estimators, which either use polynomials of higher degree or the discrete solution on a uniformly refined mesh, respectively. Such estimators have been investigated in the elliptic context, e.g., in \cite{BankSmith1993,Karakashian_2003}. Our numerical results show that these error indicators can be used to refine the mesh towards the singularity of the solution. A drawback of these estimators is that an additional global problem has to be solved in every step. 
Since the solutions to \eqref{eq:ep1}--\eqref{eq:ep2} can be discontinuous in $\mu$, the proofs developed for elliptic equations to show that the global estimator is equivalent to a locally computable quantity, see, e.g., \cite{Karakashian_2003}, do not apply.
To overcome the computational complexity of building estimators that require to solve a global problem, we propose an \textit{a posteriori} estimator based on a local averaging procedure. This cheap estimator shows a similar performance compared to the more expensive hierarchical ones mentioned before.

The outline of the rest of the manuscript is as follows. In \Cref{sec:Sec_2_ev_par_slab} we introduce notation and collect technical tools, such as trace theorems.
In \Cref{sec:Sec3_DG_derivation_analysis} we derive and analyze the discontinuous Galerkin scheme. \Cref{sec:Sec4_numerical_examples_PS_DG} presents numerical examples confirming the theoretical results of \Cref{sec:Sec3_DG_derivation_analysis}.
\Cref{sec:aposteriori} shows that our scheme works well with adaptively refined grids. We introduce here two hierarchical error estimators and one based on local post-processing. The paper closes with some conclusions in \Cref{sec:conclusion}.

\section{Preliminaries}\label{sec:Sec_2_ev_par_slab}
We denote by $L^2(\Omega)$ the usual Hilbert space of square integrable functions and denote the corresponding inner product by
\begin{align*}
    \ip{u}{v} \coloneqq \int_\Omega u(z,\mu) v(z,\mu) \,dz\,d\mu.
\end{align*}
Furthermore, we introduce the Hilbert space
\begin{align*}
	V\coloneqq\{v\in L^2(\Omega):\ \mu \dz v\in L^2(\Omega)\},
\end{align*}
which consists of square integrable functions for which the weighted derivative is also square integrable; see \cite[Section~2.2]{Agoshkov98}. We endow the space $V$ with the graph norm
\begin{align*}
	\|v\|_V^2 \coloneqq \|v\|_{L^2(\Omega)}^2 + \|\mu \dz v \|_{L^2{(\Omega)}}^2,\qquad v\in V.
\end{align*}
To treat the boundary condition \cref{eq:ep2}, let us introduce the following inner product
\begin{align*}
\boundary{u}{v} \coloneqq \int_\Gamma u v \,\mu\,d\mu \coloneqq \int_0^1 \big( u(L,\mu)v(L,\mu)+u(0,\mu)v(0,\mu)\big)\mu\, d\mu,
\end{align*}
and the corresponding space $L^2(\Gamma;\mu)$ of all measurable functions $v$ such that 
\begin{align*}
	\|v\|_{L^2(\Gamma;\mu)}^2 \coloneqq \boundary{v}{v}<\infty.
\end{align*}
According to \cite[Theorem~2.8]{Agoshkov98} and its proof, functions in $V$ have a trace on $\Gamma$ and
\begin{align}\label{eq:trace}
	\|v\|_{L^2(\Gamma;\mu)} \leq \frac{2}{\sqrt{1-\exp(-2L)}} \|v\|_{V},
\end{align}
and the trace operator mapping $V$ to $L^2(\Gamma;\mu)$ is surjective \cite[Theorem~2.9]{Agoshkov98}.
For the analysis of the numerical scheme, we provide a slightly different trace lemma.
\begin{lemma}\label{lem:trace_refined}
  Let $K=(z^l,z^r)\times(\mu^b,\mu^t)\subset \Omega$ for $0\leq z^l<z^r\leq L$ and $0\leq \mu^b<\mu^t\leq 1$. Let $F=\{z_F\}\times(\mu^b,\mu^t)$ with $z_F\in \{z^l,z^r\}$ be a vertical face of $K$. Then, for every $v\in V$ it holds that
 \begin{align*}
	\int_F |v|^2 \mu\,d\mu \leq \left(\frac{\mu^t}{z^r-z^l}\|v\|_{L^2(K)}+2\|\mu\dz v\|_{L^2(K)}\right) \|v\|_{L^2(K)}.
 \end{align*}
\end{lemma}
\begin{proof}
	Without loss of generality, we assume that $z^l=z_F=0$ and $z^r=h_z$.
	From the fundamental theorem of calculus, we obtain that
	\begin{align*}
		w(0,\mu)= w(z,\mu) - \int_0^z \dz w(y,\mu)\,d y.
	\end{align*}
	Multiplication by $\mu$, integration over $K$ and an application of the triangle inequality yields that
	\begin{align*}
		h_z\int_F |w|\mu\, d\mu \leq \int_K |w|\mu \,dz\,d\mu + \int_K \int_0^z \mu|\dz w(y,\mu)|\,dy\,dz\,d\mu.
	\end{align*}
	Setting $w=v^2$ in the previous inequality, observing that $|\mu\dz w|\leq 2|(\mu \dz v)v|$ and applying the Cauchy-Schwarz inequality shows that
	\begin{align*}
		\int_F |v|^2\mu \,d\mu \leq \int_K |v|^2\frac{\mu}{h_z} \,dz\,d\mu + 2\|\mu\dz v\|_{L^2(K)} \|v\|_{L^2(K)},
	\end{align*}
	which concludes the proof.
\end{proof}

\subsection{Weak formulation and solvability}
\label{sub: Weak formulation and solvability}
Performing the usual integration-by-parts, see, e.g., \cite{BalMaday2002,palii2020convergent}, the weak formulation of \eqref{eq:ep1}--\eqref{eq:ep2} is as follows: find $u\in V$ such that
\begin{align}\label{eq:ep_weak}
	a^e(u,v)=\ip{f}{v}+\boundary{g}{v}\quad\forall v\in V,
\end{align}
with bilinear form $a^e:V \times V \to \mathbb{R}$,
\begin{align}\label{eq:ep_bifo}
	a^e(u,v) \coloneqq \left( \frac{1}{\sigma_t} \mu\dz u, \mu \dz v \right) + (\sigma_t u,v)-(\sigma_s Pu,v) + \boundary{u}{v}.
\end{align}
Here, for ease of notation, we use the scattering operator $P:L^2(\Omega)\to L^2(\Omega)$,
\begin{align*}
	(Pu)(z,\mu) \coloneqq\int_0^1 u(z,\mu')\,d\mu'.
\end{align*}
Using the Cauchy-Schwarz inequality, we deduce that $\|Pu\|_{L^2(\Omega)}\leq\|u\|_{L^2(\Omega)}$ for $u\in L^2(\Omega)$.
Assuming 
\begin{align}\label{eq:ass1}
0\leq \sigma_s,\sigma_t\in L^\infty(0,L), \quad\sigma_t-\sigma_s\geq c>0,
\end{align}
for some $c>0$, we therefore obtain that the bilinear form $a^e$ is $V$-elliptic, and, in view of the trace theorem, cf. \cref{eq:trace}, bounded.
Similarly, for $f\in L^2(\Omega)$ and $g\in L^2(\Gamma;\mu)$, the right-hand side in \cref{eq:ep_weak} defines a bounded linear functional on $V$. Hence, there exists a unique weak solution $u\in V$ of \cref{eq:ep_weak} by the Lax-Milgram lemma, see also \cite{BalMaday2002}, \cite[Theorem~3.3]{palii2020convergent} or \cite[Section~5.3]{EggerSchlottbom12} for similar well-posedness statements.

If $f\in L^2(\Omega)$, testing \cref{eq:ep_weak} with functions in $C^\infty_0(\Omega)$ shows that $\flux{u}$ has a weak $\mu\dz$-derivative in $L^2(\Omega)$ and \cref{eq:ep1} holds a.e. in $\Omega$. In particular, $\flux{u}\in V$ and $\flux{u}$ has a trace. For $v\in V$, an integration by parts in \cref{eq:ep_weak} then shows that
\begin{align*}
    \ip{f}{v}+\boundary{g}{v}=a^e(u,v) = \ip{f}{v}+\boundary{u}{v}+\boundary{\frac{\mu}{\sigma_t}\partial_n u}{v}.
\end{align*}
Since the trace operator is surjective from $V$ to $L^2(\Gamma;\mu)$ \cite[Theorem~2.9]{Agoshkov98}, it follows that \cref{eq:ep2} holds in $L^2(\Gamma;\mu)$.
We denote the space of  solutions with data $f\in L^2(\Omega)$ and $g\in L^2(\Gamma;\mu)$ by 
\begin{align}\label{eq:Vstar}
	V_* \coloneqq \left\{ u \in V: {\frac{\mu}{\sigma_t}\dz u}\in V \right\}.
\end{align}

\section{Discontinuous Galerkin scheme}\label{sec:Sec3_DG_derivation_analysis}
In the following we will derive the numerical scheme to approximate solutions to \cref{eq:ep_weak}. After introducing a suitable partition of $\Omega$ using quad-tree grids and corresponding broken polynomial spaces, we can essentially follow the standard procedure for elliptic problems, cf. \cite{diPietroErn}. One notable difference is that we need to incorporate the weight function $\mu$ on the faces.

\subsection{Mesh and broken polynomial spaces}
In order to simplify the presentation, and subsequently the implementation, we consider quad-tree meshes \cite{Frey2008} as follows.
Let $\mathcal{T}$ be a partition of $\Omega$ such that $\sigma_t$ is constant on each element $K\in\mathcal{T}$, and that
$$K=(z_{K}^l,z_{K}^r)\times (\mu_{K}^l,\mu_{K}^r)\quad\forall \, K\in\mathcal{T},$$
for illustration see Figure \ref{fig:quadmesh}.  We denote the local mesh size by $h_{K}=z_{K}^r-z_{K}^l$. 

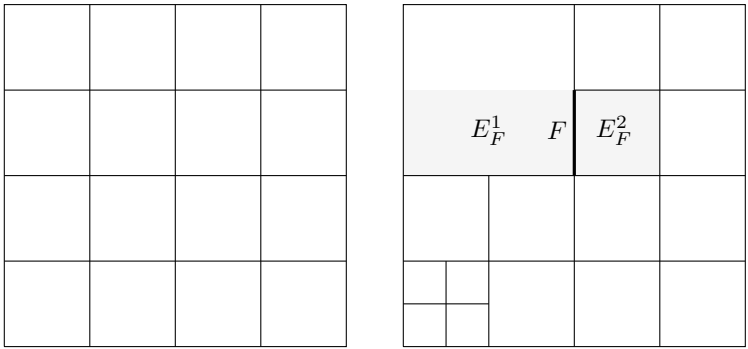
\begin{figure}[ht]
    \centering
    \begin{tikzpicture}[scale=0.75]
    \draw (0,0) -- (6,0) -- (6,6) -- (0,6) -- (0,0);
    \draw (0,1.5) -- (6,1.5);
    \draw (0,3) -- (6,3);
    \draw (0,4.5) -- (6,4.5);
    
    \draw (1.5,0) -- (1.5,6);
    \draw (3,0) -- (3,6);
    \draw (4.5,0) -- (4.5,6);

    \fill[black!5!white] (10,3) rectangle (11.5,4.5);
    \draw (10.7,3.8) node{$E^2_F$};
    \fill[black!5!white] (7,3) rectangle (10,4.5);
    \draw (8.5,3.8) node{$E^1_F$};
    \draw[very thick] (10,3) -- (10,4.5);
    \draw (9.7,3.8) node{$F$};
    
    \draw (7,0) -- (13,0) -- (13,6) -- (7,6) -- (7,0);
    \draw (7,3) -- (13,3);
    \draw (7,1.5) -- (13,1.5);
    \draw (10,4.5) -- (13,4.5);
    
    \draw (10,0) -- (10,6);
    \draw (8.5,0) -- (8.5,3);
    \draw (11.5,0) -- (11.5,6);
    
    \draw (7,0.75)--(8.5, 0.75);
    \draw (7.75,0)--(7.75, 1.5);
    \end{tikzpicture}
    \caption{Left: Uniform mesh with $16$ elements. Right: Non-uniform mesh with hanging nodes. Moreover, the two sub-elements $E^1_F$ and $E^2_F$ (shaded) for a vertical face $F\in{\Fv}^i$ (thick black line).}
    \label{fig:quadmesh}
\end{figure}

Next, let us introduce some standard notation.
Denote $\mathbb{P}_k$ the space of polynomials of one real variable of degree $k\ge 0$,
and let the broken polynomial space $V_h$ be denoted by
\begin{align}\label{eq:Vh}
	V_h \coloneqq \left\{v\in L^2(\Omega):\, v_{\mid K}\in \mathbb{P}_{\kz+1}\otimes\mathbb{P}_{\kmu}\,\,\forall K\in \mathcal{T}\right\},
\end{align}
with $\kz,\kmu\geq 0$.
Here, $\mathbb{P}_{\kz+1}\otimes\mathbb{P}_{\kmu}$ denotes the tensor product of $\mathbb{P}_{\kz+1}$ and $\mathbb{P}_{\kmu}$.
Moreover, let $V(h)\coloneqq V+V_h$.
By ${\Fv}^i$ we denote the set of interior vertical faces, that is for any $F\in{\Fv}^i$ there exist two disjoint elements 
\begin{align*}
	K_1=(z_1^l,z_1^r)\times(\mu^l_1,\mu^r_1) \text{ and } K_2=(z^l_2,z_2^r)\times(\mu^l_2,\mu^r_2)
\end{align*} 
such that $z_F=z_1^r=z^l_2$ and $F=\{z_F\}\times \big((\mu^l_1,\mu^r_1)\cap (\mu^l_2,\mu^r_2)\big)$.
For $F\in {\Fv}^i$ we define the jump and the average of $v\in V_h$ by
\[
\jump{v} \coloneqq v_{\mid K_1}(z_F,\mu)-v_{\mid K_2}(z_F,\mu),\qquad \avg{v} \coloneqq \frac{1}{2}\left(v_{\mid K_1}(z_F,\mu)+v_{\mid K_2}(z_F,\mu)\right).
\]
In order to take into account local variations in the mesh size and diffusion coefficient $1/\sigma_t$, we furthermore define the dimensionless quantity
\begin{align}\label{eq:dim_less}
    D_{F,\sigma}\coloneqq \left(\frac{1}{\sigma_{t\mid K_1}(z_F) h_{K_1}}+\frac{1}{\sigma_{t\mid K_2}(z_F) h_{K_2}}\right)^{-1}, 
\end{align}
where $h_{K_i}$, $i\in\{1,2\}$, denotes the local mesh size of the element $K_i$ in $z$-direction.
For an interior face $F\in {\Fv}^i$ with $F=\{z_F\}\times (\mu_F^b,\mu_F^t)$, which is shared by two elements $K_F^i\in \mathcal{T}$, $i=1,2$, as above, let us introduce the sub-elements
\begin{align}\label{eq:def_portion}
	E_F^i \coloneqq (z^l_i,z_i^r) \times  (\mu_F^b,\mu_F^t) \subset K_F^i.
\end{align}
We note that the inclusion in \cref{eq:def_portion} can be strict in the case of hanging nodes, see for instance \Cref{fig:quadmesh}. 

Combining \Cref{lem:trace_refined} with common inverse inequalities, cf. \cite[Sect.~4.5]{BrennerScott}, 
i.e., for any $k\geq 0$ there exists a constant $C_{ie}(k)$ such that
	\begin{align}\label{eq:inverse_inequality}
	\left(\int_{z^l}^{z^r} |v'|^2 dz\right)^{1/2} \leq \frac{\sqrt{C_{ie}(k)}}{z^r-z^l}\left(\int_{z^l}^{z^r} |v|^2 dz\right)^{1/2} \quad\forall v\in \mathbb{P}_k,
	\end{align}
we obtain the following discrete trace lemma.

\begin{lemma}[Discrete trace inequality]\label{lem:discrete_trace}
	Let $K=(z_{K}^l,z_{K}^r)\times (\mu_K^l,\mu_K^r)\in \mathcal{T}$ and let $F=\{z_F\}\times(\mu_F^b,\mu_F^t)\in \Fv$ be such that $F\subset\partial K$. Then, for any $k\geq 0$ there holds
	\begin{align*}
		\|v\|_{L^2(F;\mu)}^2 \leq \frac{C_{dt}(k)}{h_{K}} \|v\|_{L^2((z_{K}^l,z_{K}^r)\times (\mu_F^b,\mu_F^t))}^2 \quad\forall v\in \mathbb{P}_k,
	\end{align*}
	where $C_{dt}(k)=1+2\sqrt{C_{ie}(k)}$, and $C_{ie}(k)$ is the constant in \cref{eq:inverse_inequality}.
\end{lemma}
\begin{proof}
Using \Cref{lem:trace_refined} we have that
    \begin{align*}
	\int_F |v|^2 \mu\,d\mu \leq \left(\frac{\mu_F^t}{z_K^r-z_K^l}\|v\|_{L^2(K)}+2\|\mu\dz v\|_{L^2(K)}\right) \|v\|_{L^2(K)}.
 \end{align*}
 Using \cref{eq:inverse_inequality}, we estimate the weighted derivative term as follows
 \begin{align*}
     \|\mu\dz v\|_{L^2(K)}^2 = \int_{\mu_F^b}^{\mu_F^t} \mu^2 \int_{z_K^l}^{z_K^r} |\dz v|^2\,dz \,d\mu \leq  \frac{C_{ie}(k)}{z_K^r-z_K^l} \int_{\mu_F^b}^{\mu_F^t} \mu^2 \int_{z_K^l}^{z_K^r} |v|^2\,dz \,d\mu.
 \end{align*}
 Using that $h_K=z_K^r-z_K^l$ and $\mu\leq 1$, we thus obtain that
 \begin{align*}
    \int_F |v|^2 \mu\,d\mu \leq \frac{1+2\sqrt{C_{ie}(k)}}{h_K}\|v\|_{L^2(K)}^2,
 \end{align*}
 which concludes the proof.
\end{proof}

\begin{remark}\label{rem:computing_Cie}
The value of $C_{ie}(k)$ of the inverse inequality in \cref{eq:inverse_inequality} can be computed by solving a small eigenvalue problem of dimension $k+1$, which is obtained by transforming \cref{eq:inverse_inequality} to the unit interval. In fact, $C_{ie}(k)$ is the maximal eigenvalue of 
\begin{align*}
    D v= \lambda M v,
\end{align*}
where $D_{i,j}=\int_{0}^1 \varphi_i'(\hat z)\varphi_j'(\hat z)\,d\hat z$ and $M_{i,j}=\int_{0}^1 \varphi_i(\hat z)\varphi_j(\hat z)\,d\hat z$  for a basis $\{\varphi_i\}_{i=0}^k$ of the space of polynomials of degree at most $k$ on the unit interval. Explicit bounds for $C_{ie}(k)$, which are optimal for $k=1,2$, are given in \cite{Chen2013}.
\end{remark}

\subsection{Derivation of the DG scheme}
In order to extend the bilinear form defined in \cref{eq:ep_bifo} to the broken space $V_h$, we denote with $\dz^h$ the broken derivative operator such that
\begin{align*}
	\left(\frac{\mu^2}{\sigma_t}\dz^h u_h,\dz^h v_h\right)=\sum_{K\in \mathcal{T}}\int_{K}\frac{\mu^2}{\sigma_t}\dz u_h \dz v_h \,dz\,d\mu
\end{align*}
for $u_h,v_h\in V_h$. In view of \cref{eq:ep_weak}, let us then introduce the bilinear form
\begin{align*}
    a^e_h(u,v) \coloneqq \ip{\frac{\mu^2}{\sigma_t}\dz^h u}{\dz^h v} + \ip{\sigma_t u}{v}-\ip{\sigma_s Pu}{v} + \boundary{u}{v},
\end{align*}
which is defined on $V(h)$. Note that $a^e$ and $a^e_h$ coincide on $V$.
In order to obtain a consistent bilinear form, $a^e_h$ needs to be modified. We follow \cite[Chapter~4]{diPietroErn} to determine the required modification. 
Choosing $w\in V_{*h} \coloneqq V_*+V_h$ and $v_h\in V_h$, integration by parts in $z$ shows that
\begin{align*}
&\sum_{K\in \mathcal{T}}\int_{K}\frac{\mu^2}{\sigma_t}\dz^h w \dz^h v_h \,dz\,d\mu + \left(\dz^h \left(\frac{\mu^2}{\sigma_t}\dz^h w\right)\right)  v_h \,dz\,d\mu\\
=&\sum_{K\in \mathcal{T}} \int_{\mu_K^l}^{\mu_K^r} \left(\frac{\mu}{\sigma_t(z_K^r)}\dz^hw(z_K^r)v_h(z_K^r)-\frac{\mu}{\sigma_t(z_K^l)}\dz^hw(z_K^l)v_h(z_K^l)\right)\mu\,d\mu\\
=& \sum_{F\in{\Fv}^b}\int_{F} \frac{\mu}{\sigma_t}\partial_n w v_h\, \mu\,d\mu+\sum_{F\in{\Fv}^i}\int_{F} \jump{\frac{\mu}{\sigma_t}\dz^h w v_h}\, \mu\,d\mu\\
=& \sum_{F\in{\Fv}^b}\int_{F} \frac{\mu}{\sigma_t}\partial_n w v_h\, \mu\,d\mu+\sum_{F\in{\Fv}^i}\int_{F} \left(\avg{\fluxh{w}}\jump{v_h}+\jump{\fluxh{w}}\llbrace v_h \rrbrace\right) \mu\,d\mu,
\end{align*}
where we used the identity $\jump{\frac{\mu}{\sigma_t}\dz^h w v_h}=\avg{\fluxh{w}}\jump{v_h}+\jump{\fluxh{w}}\llbrace v_h \rrbrace$ in the last step, see \cite[p. 123]{diPietroErn}.
%
Hence, for any solution $u\in V_*$ to \eqref{eq:ep1}--\eqref{eq:ep2} and $v\in V_h$ we have that
\begin{align*}
    a_h^e(u,v)= \ip{f}{v}+\boundary{g}{v} + \sum_{F\in {\Fv}^i} \int_F \left(\avg{\frac{\mu}{\sigma_t}\dz^h u}\jump{v} + \jump{\frac{\mu}{\sigma_t}\dz^h u}\llbrace v \rrbrace\right)\mu\,d\mu.
\end{align*}
Since $ \jump{\frac{\mu}{\sigma_t}\dz^h u}=0$ for all $F \in {\Fv}^i$ by $z$-continuity of the flux of $u\in V_*$, we arrive at the identity
\[
a_h^e(u,v)= \ip{f}{v}+\boundary{g}{v} + \sum_{F\in {\Fv}^i} \int_F \avg{\frac{\mu}{\sigma_t}\dz^h u}\jump{v} \mu\,d\mu.\]
Hence, a consistent bilinear form is given by
\begin{align*}
    a^c_h(u,v)\coloneqq a^e_h(u,v) - \sum_{F\in {\Fv}^i} \int_F \avg{\frac{\mu}{\sigma_t}\dz^h u}\jump{v}\mu\,d\mu,
\end{align*}
which, for $V_{*h} \coloneqq V_*+V_h$, is well-defined on $V_{*h}\times V_h$.
Using that $\jump{u}=0$ on $F\in {\Fv}^i$ for any $u\in V$, we arrive at the following symmetric and consistent bilinear form
\begin{align*}
    a^{cs}_h(u,v) \coloneqq a^e_h(u,v) - \sum_{F\in {\Fv}^i} \int_F \left(\avg{\frac{\mu}{\sigma_t}\dz^h u}\jump{v}+\avg{\frac{\mu}{\sigma_t}\dz^h v}\jump{u} \right)\mu\,d\mu,
\end{align*}
which is again well-defined on $V_{*h}\times V_h$.
We note that the summation over the vertical faces on the boundary $\Gamma$ is included in the term $\boundary{u}{v}$ in $a^e_h$. The stabilized bilinear form is then defined on $V_{*h}\times V_h$ by
\begin{align}
    \label{eq: SIP bilinear form}
	    a_h(u,v) \coloneqq a_h^{cs}(u,v) + \sum_{F\in {\Fv}^i} \frac{\alpha_F}{D_{F,\sigma}} \int_F \jump{u} \jump{v} \mu\, d\mu,
\end{align}
with $D_{F,\sigma}$ defined in \cref{eq:dim_less}
and with positive penalty parameter $\alpha_F>0$, which will be specified below. Since $\jump{u}=0$ on any $F\in {\Fv}^i$ and $u\in V$, it follows 
that $a_h$ is consistent, i.e., for $u\in V_*$ it holds that
\begin{align}\label{eq:consistency}
	a_h(u,v_h) = a^e(u,v_h)\qquad \forall v\in V_h.
\end{align}
The discrete variational problem is formulated as follows: Find $u_h\in V_h$ such that
\begin{align}
    \label{eq:DG}
	a_h(u_h,v_h) = \ip{f}{v_h} + \boundary{g}{v_h} \quad\forall v_h\in V_h.
\end{align}

\subsection{Analysis}
For the analysis of \eqref{eq:DG}, let us introduce mesh-dependent norms 
\begin{subequations}
\begin{align}
    \label{eq: norm Vh}
    \|v\|_{V_h}^2 &\coloneqq a_h^e(v,v) + \sum_{F\in {\Fv}^i} D_{F,\sigma}^{-1} \|\jump{v}\|_{L^2(F;\mu)}^2, \quad v \in V(h),\\
    \label{eq: norm star}
	\|v\|_{*}^2 &\coloneqq \|v\|_{V_h}^2 + \sum_{F\in {\Fv}^i} \frac{D_{F,\sigma}}{C_{dt}(\kz)} \left\Vert\avg{\fluxh{v}}\right\Vert_{L^2(F;\mu)}^2,\quad v\in V_{*h}.
 \end{align}
\end{subequations}
In order to show discrete stability and boundedness of $a_h$, we will use the following auxiliary lemma.
\begin{lemma}[Auxiliary lemma]\label{lem:technical}
	Let $F\in {\Fv}^i$ be shared by the elements $K_F^1,K_F^2\in\mathcal{T}$.
	Then, for $w\in V_h$ and $v\in V(h)$ it holds that
	\begin{align*}
        \int_F \avg{\fluxh{w}} \jump{v}\mu\, d\mu\leq \frac{\sqrt{C_{dt}(\kz)}}{2\sqrt{D_{F,\sigma}}} \norm{\frac{\mu}{\sqrt{\sigma_t}}\dz^h w}_{L^2(E_F^1 \cup E_F^2)}  \norm{\jump{v}}_{L^2(F;\mu)},
	\end{align*}
	with $C_{dt}(\kz)$ from \Cref{lem:discrete_trace} and sub-elements $E_F^i$, $i=1,2$, defined in \cref{eq:def_portion}.
\end{lemma}
\begin{proof}
	By definition of the average, we have that
	\begin{align*}
		\int_F \avg{\fluxh{w}} \jump{v}\mu\, d\mu=\frac{1}{2} \int_F {\fluxh[1]{w}} \jump{v}\mu \, d\mu + \frac{1}{2}\int_F {\fluxh[2]{w}} \jump{v}\mu\, d\mu,
	\end{align*}
	where $w_1,w_2$ and $\sigma^1_t,\sigma_t^2$ denote the restrictions of $w$ and $\sigma_t$ to $K_F^1$ and $K_F^2$, respectively.
    To estimate the first integral on the right-hand side, we employ the Cauchy-Schwarz inequality to obtain
	\begin{align*}
		 \int_F {\fluxh[1]{w}} \jump{v}\mu \, d\mu \leq \norm{\fluxh[1]{w}}_{L^2(F;\mu)} \norm{\jump{v}}_{L^2(F;\mu)}
		 \leq \frac{\sqrt{C_{dt}(\kz)}}{\sqrt{\sigma_t^1 h_{K_F^1}}} \norm{ \fluxhs[1]{w}}_{L^2(E_F^1)}\norm{\jump{v}}_{L^2(F;\mu)},
	\end{align*}
	where we used \Cref{lem:discrete_trace} applied to $\flux{w_1}$, which is a piecewise polynomial of degree $\kz$ in $z$.
	A similar estimate holds for the second integral. Using the Cauchy-Schwarz inequality, we then obtain that
	\begin{align*}
		\int_F \avg{\flux{w}} \jump{v}\mu\, d\mu \leq  \frac{\sqrt{C_{dt}(\kz)}}{2} \norm{\fluxhs{w}}_{L^2(E_F^1 \cup E_F^2)} \sqrt{\frac{1}{\sigma_t^1 h_{K_F^1}}+\frac{1}{\sigma_t^2 h_{K_F^2}}} \norm{\jump{v}}_{L^2(F;\mu)},
	\end{align*}
	which, in view of \cref{eq:dim_less}, concludes the proof.
\end{proof}
The auxiliary lemma allows to bound the consistency terms in $a_h$, which gives discrete stability of $a_h$.
\begin{lemma}[Discrete stability]
    \label{lem:discrete_stability}
    For any $v\in V_h$ it holds that
    \begin{equation*}
        a_h(v,v)\geq \frac{1}{2} \|v\|_{V_h}^2
    \end{equation*}
    provided that $\alpha_F\geq 1/2+C_{dt}(\kz)$ with constant $C_{dt}(\kz)$ given in \Cref{lem:discrete_trace}.
\end{lemma}
\begin{proof}
	Let $v_h\in V_h$, and consider 
	\begin{align*}
		a_h(v_h,v_h) = a^e_h(v_h,v_h) - 2\sum_{F\in {\Fv}^i} \int_F \avg{\flux{v_h}} \jump{v_h}\mu\, d\mu + \sum_{F\in {\Fv}^i} \frac{\alpha_F}{D_{F,\sigma}} \int_F \jump{v_h}^2 \mu\, d\mu.
	\end{align*}
	Using \Cref{lem:technical}, and the fact that each sub-element $E_F^i$ touches at most two interior vertical faces,
	an application of the  Cauchy-Schwarz yields for any $\epsilon>0$,
	\begin{align*}
		2\sum_{F\in {\Fv}^i} \int_F \avg{\flux{v_h}} \jump{v_h}\mu \, d\mu \leq \epsilon \norm{\fluxhs{v_h}}^2_{L^2(\Omega)} + \sum_{F\in {\Fv}^i} \frac{C_{dt}}{2\epsilon D_{F,\sigma}}\int_F \jump{v_h}^2\mu\,  d\mu.
	\end{align*}
	Hence, by choosing $\epsilon=1/2$,
	\begin{align*}
		a_h(v_h,v_h)\geq \frac{1}{2} a^e_h(v_h,v_h)  + \sum_{F\in {\Fv}^i} \frac{\alpha_F-C_{dt}}{D_{F,\sigma}}\int_F \jump{v_h}^2\mu \, d\mu,
	\end{align*}
	from which we obtain the assertion.
\end{proof}
Discrete stability implies that the scheme \eqref{eq:DG} is well-posed, cf. \cite[Lemma~1.30]{diPietroErn}.
\begin{theorem}[Discrete well-posedness]\label{thm:wellposedness}
	Let $\alpha_F\geq 1/2+C_{dt}(\kz)$ with constant $C_{dt}(\kz)$ given in \Cref{lem:discrete_trace}.
    Then for any $f\in L^2(\Omega)$ and $g\in L^2(\Gamma;\mu)$ there exists a unique solution $u_h\in V_h$ of the discrete variational problem \eqref{eq:DG}.
\end{theorem}
\begin{proof}
 The space $V_h$ is finite-dimensional. Hence, \Cref{lem:discrete_stability} implies the assertion.
\end{proof}
To proceed with an abstract error estimate, we need the following boundedness result.
\begin{lemma}[Boundedness]\label{lem:boundedness}
For any $u\in V_{*h}$ and $v\in V_h$ it holds that
\begin{align*}
    a_h(u,v) \leq (C_{dt}+\alpha_F)\|u\|_{*} \|v\|_{V_h},
\end{align*}
where $\alpha_F$ is as in \Cref{lem:discrete_stability}.
\end{lemma}
\begin{proof}
	We have that
	\begin{align*}
		a_h(u,v) &= a^e_h(u,v) - \sum_{F\in {\Fv}^i} \int_F \avg{\frac{\mu}{\sigma_t}\dz^h u}\jump{v} \mu \,d\mu - \sum_{F\in {\Fv}^i} \int_F\avg{\frac{\mu}{\sigma_t}\dz^h v}\jump{u} \mu\,d\mu \\&+ \sum_{F\in {\Fv}^i} \frac{\alpha_F}{D_{F,\sigma}} \int_F \jump{u} \jump{v} \mu \,d\mu.
	\end{align*}
	 The first two terms can be estimated using the Cauchy-Schwarz inequality as follows
	\begin{align*}
		a^e_h(u,v) & \leq a^e_h(u,u)^{1/2} a^e_h(v,v)^{1/2}, \\
		\sum_{F\in {\Fv}^i} \int_F \avg{\frac{\mu}{\sigma_t}\dz^h u}\jump{v} \mu \,d\mu
        & \leq \sum_{F\in {\Fv}^i} \norm{\avg{\fluxh{u}}}_{L^2(F;\mu)} \norm{\jump{v}}_{L^2(F;\mu)}.
	\end{align*}
	For the third term, we use \Cref{lem:technical} to obtain
	\begin{align*}
		\sum_{F\in {\Fv}^i} \int_F \avg{\fluxh{v}} \jump{u}\mu\, d\mu 
		\leq 
		\sum_{F\in {\Fv}^i} \frac{\sqrt{C_{dt}}}{2\sqrt{D_{F,\sigma}}} \norm{\fluxhs{v}}_{L^2(E_F^1\cup E_F^2)}\norm{\jump{u}}_{L^2(F;\mu)}.
	\end{align*}
		To separate the terms that include $u$ and $v$, respectively, we apply the Cauchy-Schwarz inequality once more and use again that each sub-element $E_F^i$ touches at most two interior faces, to arrive at
		\begin{align*}
			a_h(u,v) \leq \left(a_h^e(u,u) + \sum_{F\in {\Fv}^i} \frac{D_{F,\sigma}}{C_{dt}}\norm{\avg{\fluxh{u}}}_{L^2(F;\mu)}^2 + \frac{{C_{dt}+\alpha_F}}{D_{F,\sigma}} \norm{\jump{u}}_{L^2(F;\mu)}^2 \right)^{1/2}\\
			\left(a_h^e(v,v) + \frac{1}{2}\norm{\fluxhs{v}}^2_{L^2(\Omega)} + \sum_{F\in {\Fv}^i}  \frac{{C_{dt}+\alpha_F}}{D_{F,\sigma}} \norm{\jump{v}}_{L^2(F;\mu)}^2 \right)^{1/2},
		\end{align*}
		which concludes the proof as $C_{dt}+\alpha_F\geq 3/2$.
\end{proof}
Before continuing, an inspection of the previous proof shows that we have the following corollary stating boundedness of $a_h$ on $V_h$.
\begin{corollary}[Discrete boundedness]
    \label{cor:discrete_boundedness}
    For any $u,v\in V_h$ it holds that
    \begin{align*}
        a_h(u,v) \leq (C_{dt}+\alpha_F)\|u\|_{V_h} \|v\|_{V_h},
    \end{align*}
    where $\alpha_F$ is as in \Cref{lem:discrete_stability}.
\end{corollary}

Combining consistency, stability and boundedness ensures that the discrete solution $u_h$ to \eqref{eq:DG} yields a quasi-best approximation to $u$, cf. \cite[Theorem~1.35]{diPietroErn}.
\begin{theorem}[Error estimate]\label{thm:error_estimate}
    Let $f\in L^2(\Omega)$ and $g\in L^2(\Gamma;\mu)$, and denote $u\in V_*$ the solution to \eqref{eq:ep1}--\eqref{eq:ep2} and $u_h\in V_h$ the solution to \eqref{eq:DG}.
 Then the following error estimate holds true
    \[ \|u-u_h\|_{V_h} \leq \left(1+2( C_{dt}(\kz)+\alpha_F)\right)\inf_{v_h\in V_h}\|u-v_h\|_{*},\]
	provided that $\alpha_F\geq 1/2+C_{dt}(\kz)$.
\end{theorem}
\begin{remark}
    Note that $C_{dt}(k)$ is monotonically increasing in $k$. Thus, replacing $C_{dt}(\kz)$ by $C_{dt}(0)=1$ in \cref{eq: norm star} yields a norm that is independent of $\kz$ and that is an upper bound for $\|\cdot\|_{*}$. Hence, the error estimate in \Cref{thm:error_estimate} deteriorates for increasing $\kz$ only through the constant pre-multiplying the best-approximation error.
\end{remark}
\begin{remark}
    \label{rem:convergence_rate}
    Assuming that the exact solution is sufficiently regular, say $u\in H^{k+1}(\Omega)$, denoting $h$ the maximal mesh-size, and setting $k=k_z=k_\mu$, standard interpolation error estimates yield a convergence rate of $O(h^{k+1})$ for $ \|u-u_h\|_{V_h}$, see \cite[Lemmata 1.58, 1.59, p. 31-32]{diPietroErn} and \cite[Corollary 4.22, p. 132]{diPietroErn}.
\end{remark}
\begin{remark}
    In view of \Cref{rem:computing_Cie}, the value of $C_{dt}(\kz)$ can be computed explicitly once $C_{ie}$ is known. Hence, we can give an explicit value for the penalization parameter $\alpha_F$ such that the discontinuous Galerkin scheme \cref{eq:DG} is well-posed and the error enjoys the bound given in \Cref{thm:error_estimate}. We note that we choose here $\alpha_F$ to be the same for all interior faces. Moreover, the choice of $\alpha_F$ is independent of the partition $\T$ and the mean-free path $1/\sigma_t$; while the dependence on the mean-free path is explicit through $D_{F,\sigma}$, which might be exploited if the behavior of the scheme is investigated in the diffusion limit where the mean-free path tends to zero. Let us refer to \cite{guermond2014discontinuous} for a detailed discussion about issues of DG schemes for radiative transfer in the diffusion limit.
\end{remark}
\begin{remark}\label{rem:parameter_dependent_bilinearform}
    Instead of using the symmetric bilinear form $a_h^{cs}$ to define $a_h$ in \cref{eq: SIP bilinear form}, we may use the more general bilinear form
    $$
        a^e_h(u,v) - \sum_{F\in {\Fv}^i} \int_F \left(\avg{\frac{\mu}{\sigma_t}\dz^h u}\jump{v}+\lambda \avg{\frac{\mu}{\sigma_t}\dz^h v}\jump{u} \right)\mu\,d\mu,
    $$
    with parameter $\lambda\in[-1,1]$, cf. \cite{DAWSON2004}. The choice $\lambda=1$ leads to $a_h^{cs}$, while the choices $\lambda=0$ or $\lambda=-1$ yield incomplete interior penalty and the non-symmetric interior penalty discontinuous Galerkin schemes, respectively, see also \cite{Houston_2002,Riviere2001} for the case $\lambda=-1$. We note that for $\lambda=-1$, the terms involving the face integral vanish for $u=v$, and hence coercivity can be proven straight-forward. 
    However, since symmetry is lost, improved $L^2$-convergence rates for sufficiently smooth solutions do not hold in general, see \cite{Arnold_2002} and \Cref{tab:L2-convergence_uniform_lambda_minus1} below. Moreover, the numerical solution of large non-symmetric linear systems can be more difficult than in the symmetric case. We mention that the results in this section can be extended with minor modifications to the general case $-1\leq \lambda\leq 1$.
\end{remark}

\section{Numerical examples}\label{sec:Sec4_numerical_examples_PS_DG}
In the following we confirm the theoretical statements about stability and convergence of \Cref{sec:Sec3_DG_derivation_analysis} numerically.
Let $\sigma_s=1/2$ and $\sigma_t=1$ and the width of the slab be $L=1$. We then define the source terms $f$ and $g$ in \eqref{eq:ep1}--\eqref{eq:ep2} such that the exact solution is given by the following function
\begin{align}\label{eq:u_smooth}
	u(z,\mu) = \big(1+\exp(-\mu)\big)\chi_{\{\mu>1/2\}}(\mu) \exp(-z^2).
\end{align}
Here, $\chi_{\{\mu>1/2\}}(\mu)$ denotes the indicator function of the interval $(1/2,1)$, i.e., $u$ is discontinuous in $\mu=1/2$, but note that $u\in V_*$.
We compute the DG solution $u_h$ of \eqref{eq:DG} on a sequence of uniformly refined meshes, initially consisting of $16$ elements, see \Cref{fig:quadmesh}. Hence, the discontinuity in $u$ is resolved by the mesh. 

For our computations we use the spaces $V_h$ with $\kz$ and $\kmu$ in \cref{eq:Vh}, that is piecewise polynomials of degree $\kmu$ in $\mu$ and piecewise polynomials of degree $\kz+1$ in $z$.
The value of $C_{ie}(\kz)$ of the inverse inequality in \cref{eq:inverse_inequality} is computed numerically by solving a small eigenvalue problem of dimension $\kz+1$, see \Cref{rem:computing_Cie}.
 
For the numerical solution of the resulting linear systems, we a usual fixed-point iteration \cite{AdamsLarsen02}: Introducing the auxiliary bilinear form $b_h(u,v)=a_h(u,v)-(\sigma_s Pu,v)$, the fixed-point iteration maps $u_h^{n}$ to $u_h^{n+1}$ by solving
\begin{align}\label{eq:si}
    b_h(u_h^{n+1},v) = (\sigma_s P u_h^n,v)+\ip{f}{v}+\boundary{g}{v} \quad\forall v\in V_h.
\end{align}
The fixed-point iteration converges linearly with a rate $\sigma_s/\sigma_t$ \cite{AdamsLarsen02}, which is bounded by $1/2$ in this example. The iteration is stopped as soon as $\|u_h^{n+1}-u_h^n\|_{L^2(\Omega)}<10^{-10}$. For acceleration of the source iteration by preconditioning see also \cite{AdamsLarsen02,palii2020convergent}. The matrix representation of $b_h$ has a block structure for the uniformly refined meshes considered in this section, and its inverse can be applied efficiently via LU factorization.

\Cref{tab:convergence_uniform} shows the $V_h$ norm of the error $u-u_h$ between the exact and the numerical solution for $k = k_z = k_\mu$. 
For fixed $k$, we observe a convergence rate of $k+1$ under mesh refinement, which is expected from the smoothness of $u$ per element and \Cref{rem:convergence_rate}.
In particular, inspecting \Cref{tab:convergence_uniform} by rows, we notice linear convergence for $k=0$, quadratic convergence for $k=1$, and so on.
\begin{table}[ht!]
\centering\small\setlength\tabcolsep{0.75em}
\caption{Error $\|u-u_h\|_{V_h}$ for different local polynomial degrees with $k=\kz=\kmu$, see \cref{eq:Vh}, and uniformly refined meshes with $N$ elements and solution $u$ defined in \cref{eq:u_smooth}.\label{tab:convergence_uniform}}
\begin{tabular}{c r r r r r r r}
\toprule
 $N$ 				& 		16 	& 64 		& 256 		& 1\,024	& 4\,096 & 16\,384 & 65\,536\\
 \hline
$k=0$	&  7.07e-02 & 3.53e-02 & 1.76e-02 & 8.81e-03 & 4.40e-03 & 2.20e-03 & 1.10e-03 \\
$k=1$	&  5.51e-03 & 1.38e-03 & 3.44e-04 & 8.60e-05 & 2.15e-05 & 5.37e-06 & 1.34e-06 \\
$k=2$	&  2.77e-04 & 3.47e-05 & 4.33e-06 & 5.41e-07 & 6.77e-08 & 8.46e-09 & 1.06e-09 \\ 
$k=3$	&  1.38e-05 & 8.69e-07 & 5.44e-08  & 3.40e-09 & 2.16e-10 & 4.20e-11 & 4.16e-11 \\
\bottomrule
\end{tabular}
\end{table}

Since the coefficients are smooth, we may expect higher order convergence in the $L^2$-norm for the symmetric formulation if $\kmu=\kz+1$, see \Cref{rem:parameter_dependent_bilinearform}.
In \Cref{tab:L2-convergence_uniform} and \Cref{tab:L2-convergence_uniform_lambda_minus1} we compare the symmetric interior penalty method ($\lambda=1$) with its non-symmetric counterpart ($\lambda=-1$), with $\lambda$ introduced in \Cref{rem:parameter_dependent_bilinearform}.
\Cref{tab:L2-convergence_uniform} shows that, for fixed $\kmu$, the $L^2$-error decays upon mesh refinement at an improved rate of $O(h^{\kmu+1})$ for the symmetric interior penalty method. 
This improved convergence rate can also be observed for the non-symmetric interior penalty method if the employed polynomial degrees are odd, while the suboptimal rate  $O(h^{\kmu})$ can be observed if the used polynomial degrees are even, cf. \cite{Hozman2009} for a similar observation on the convergence rates for the unsymmetric interior penalty method in the context of non-stationary convection diffusion problems.

\begin{table}[ht!]
\centering\small\setlength\tabcolsep{0.75em}
\caption{$L^2$-error $\|u-u_h\|_{L^2(\Omega)}$ for different local polynomial degrees with $k=\kz$ and $\kmu=\kz+1$, see \cref{eq:Vh}, and uniformly refined meshes with $N$ elements and solution $u$ defined in \cref{eq:u_smooth}.\label{tab:L2-convergence_uniform}}
\begin{tabular}{c r r r r r r r}
\toprule
 $N$ 				& 		16 	& 64 		& 256 		& 1\,024	& 4\,096 & 16\,384 & 65\,536\\
 \hline
$k=0$	& 5.75e-03 & 1.49e-03 & 3.78e-04 & 9.46e-05 & 2.37e-05 & 5.92e-06 & 1.48e-06 \\ 
$k=1$	& 2.13e-04 & 2.60e-05 & 3.22e-06 & 4.02e-07 & 5.02e-08 & 6.27e-09 & 7.84e-10 \\ 
$k=2$	& 9.43e-06 & 6.03e-07 & 3.79e-08 & 2.37e-09 & 1.53e-10 & 3.86e-11 & 3.79e-11 \\ 
$k=3$   & 3.11e-07 & 9.64e-09 & 3.03e-10 & 3.85e-11 & 3.75e-11 & 3.75e-11 & 3.92e-11 \\ 
\bottomrule
\end{tabular}
\end{table}

\begin{table}[ht!]
\centering\small\setlength\tabcolsep{0.75em}
\caption{$L^2$-error $\|u-u_h^\lambda\|_{L^2(\Omega)}$ for different local polynomial degrees with $k=\kz$ and $\kmu=\kz+1$, see \cref{eq:Vh}, and uniformly refined meshes with $N$ elements and solution $u$ defined in \cref{eq:u_smooth}. Here, the unsymmetric interior penalty method with $\lambda=-1$ described in \Cref{rem:parameter_dependent_bilinearform} is used to compute the numerical solution $u_h^\lambda$. \label{tab:L2-convergence_uniform_lambda_minus1}}
\begin{tabular}{c r r r r r r r}
\toprule
 $N$ 				& 		16 	& 64 		& 256 		& 1\,024	& 4\,096 & 16\,384 & 65\,536\\
 \hline
$k=0$ & 4.46e-03 & 1.10e-03 & 2.74e-04 & 6.84e-05 & 1.71e-05 & 4.27e-06 & 1.07e-06 \\ 
$k=1$ & 5.26e-04 & 1.17e-04 & 2.80e-05 & 6.93e-06 & 1.73e-06 & 4.31e-07 & 1.08e-07 \\ 
$k=2$ & 1.08e-05 & 6.67e-07 & 4.16e-08 & 2.59e-09 & 1.65e-10 & 3.84e-11 & 3.82e-11 \\ 
$k=3$ & 6.27e-07 & 3.31e-08 & 1.96e-09 & 1.30e-10 & 3.89e-11 & 3.74e-11 & 3.93e-11\\ 
\bottomrule
\end{tabular}
\end{table}

\section{Adaptivity}
\label{sec:aposteriori}
In this section we show, by examples, that hierarchical error estimators, see, e.g., \cite{Karakashian_2003} for the elliptic case, as well as estimators based on averaging the approximate solutions are a possible choice to adaptively construct optimal partitions $\T$ of $\Omega$ to approximate non-smooth solutions to \cref{eq:ep1}. In particular, we show how adaptive mesh refinement is beneficial if the discontinuity of the solution is not resolved by the mesh. 
To highlight the dependency on the partition $\T$ of $\Omega$ and on the polynomial degree, we will write $u_{\T,k}$ instead of $u_{h}$ for the solution of the discontinuous Galerkin scheme \eqref{eq:DG}. 
Similarly, assuming $k\coloneqq\kz=\kmu$, we write $V_{\T,k}$ for the corresponding approximation space instead of $V_h$, see \cref{eq:Vh}.
Let $\T'$ be another partition of $\Omega$ such that $\T\subset\T'$, and let $k'\geq k$. 
Denoting $\|\cdot\|$ some norm defined on $V+V_{\T,k}+V_{\T',k'}$, and supposing the saturation assumption, which has been used, e.g., also in \cite{BankSmith1993},
\begin{align*}
	\|u-u_{\T',k'}\| \leq \gamma\|u-u_{\T,k}\|,
\end{align*}
for some universal constant $\gamma<1$, 
we obtain the equivalence between the approximation error and the estimator
    $\zeta \coloneqq u_{\T', k'} - u_{\T, k}$,
i.e., 
\begin{align*}
	(1+\gamma)^{-1}\|\zeta\| \leq \norm{u-u_{\T,k}} \leq (1-\gamma)^{-1}\norm{\zeta}.
\end{align*}
For a justification of the saturation assumption in the context of elliptic problems we refer to \cite{Cartensten_Gallisti2016}. 
In the following numerical experiments, we use the norm $\normc_{\T}$, defined as
\begin{align}
    \label{eq: broken H1 norm}
    \norm{v}^2_{\T} := \sum_{K\in\T} \left( \norm{\mu \dz^h v}_{L^2(K)}^2 + \norm{v}_{L^2(K)}^2 \right) \quad \forall v \in V_{*h},
\end{align}
to investigate the behaviour of two hierarchical error indicators for different test cases. The local error contributions are then given by
\begin{align*}
    \eta_K \coloneqq (\|\mu\dz^h \zeta\|_{L^2(K)}^2 + \|\zeta\|_{L^2(K)}^2)^{1/2},
\end{align*}
where $K\in\T$.
The mesh is then refined by a D\"orfler marking strategy \cite{Doerfler_1996,Verfuerth2013}, that is all elements in the set $\mathcal{K}\subset\mathcal{T}$ are refined, where $\mathcal{K}\subset\mathcal{T}$ is the set of smallest cardinality such that
\begin{align}
    \label{eq: Doerfler marking strategy}
	\sum_{K\in\mathcal{K}}\eta_K^2 > \theta \sum_{K\in \mathcal{T}}\eta_K^2,
\end{align}
where $0<\theta\leq 1$ is the bulk-chasing parameter. Differently from the previous section, we assume $\sigma_t = 1$ and $\sigma_s = 0$. Moreover, we consider two different manufactured solution $u_1$ and $u_2$ given by
\begin{align}
    \label{eq: Test case 6}
    u_1(z, \mu) & \coloneqq (\mu^2 + z^2)^{1/4}, \\
    \label{eq: Test case 3}
    u_2(z, \mu) & \coloneqq \left(1 + \chi_{\{\mu>1/\sqrt{2}\}}(\mu) \right)\exp(-z^2).
\end{align}
The choice of $1/\sqrt{2}$ in the indicator function ensures that the corresponding line discontinuity of $u_2$ is never resolved by our mesh. Furthermore, $\mu\dz u_1$ is bounded and vanishes in $(0,0)$. In particular, we note that $u_1,u_2\in V_*$. 
In the following we report on numerical examples using the D\"orfler parameter $\theta \coloneqq 0.75$. We note that we obtained similar results for the choice $\theta=0.3$.

\subsection{Hierarchical \textit{p}-error estimator}
\label{sec: Hierarchical p error estimator}
Setting $\T' \coloneqq\T$ and $k' \coloneqq k+1$, the hierarchical $p$-error estimator is defined as
\begin{align}
    \label{eq:hierarchical_error_estimator1}
	\zeta_p \coloneqq u_{\T,k}-u_{\T,k+1}.
\end{align}
We note that $\alpha_F\coloneqq 1/2+C_{dt}(k+1)$ is used for the stabilization parameter to obtain both numerical solutions $u_{\T,k}$ and $u_{\T,k+1}$.

\Cref{fig: TC6-H1 norm p-estimator} and \Cref{fig: TC3-H1 norm p-estimator} show the convergence rates for adaptively refined meshes using the $\zeta_p$ indicator, for different values of the polynomial degree $k$. We observe that for the manufactured solution $u_1$, which has a point singularity in the origin, the indicator follows tightly the curve of the actual error. Moreover, the error decays at the optimal rate $1/\sqrt{N^{k+1}}$, with  $N$ denoting the number of degrees of freedom in $V_{\T,k}$, also shown for comparison. 

For the manufactured solution $u_2$ with line discontinuity defined in \cref{eq: Test case 3}, the convergence behavior is different.
For $k=0$, the error and the error estimator stay rather close, and follow the curve for the optimal rate. 
For $k\geq 1$ the rate is sub-optimal, which is expected from a counting argument. 
Moreover, also the error estimator is not as close to the true error anymore, compared to the test case with $u_1$. 

\begin{figure}\centering
	\includegraphics[width=.49\textwidth]{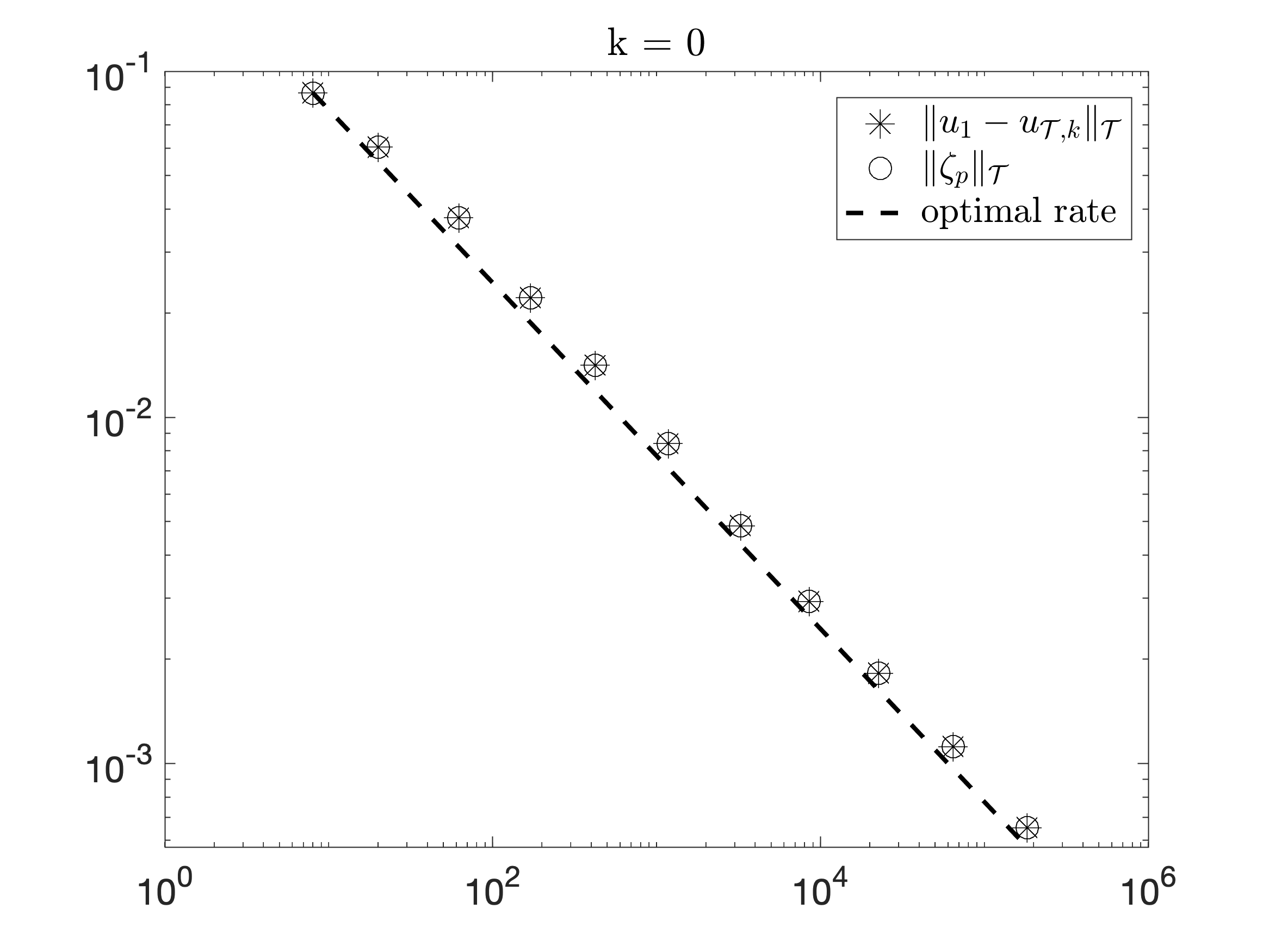}
	\includegraphics[width=.49\textwidth]{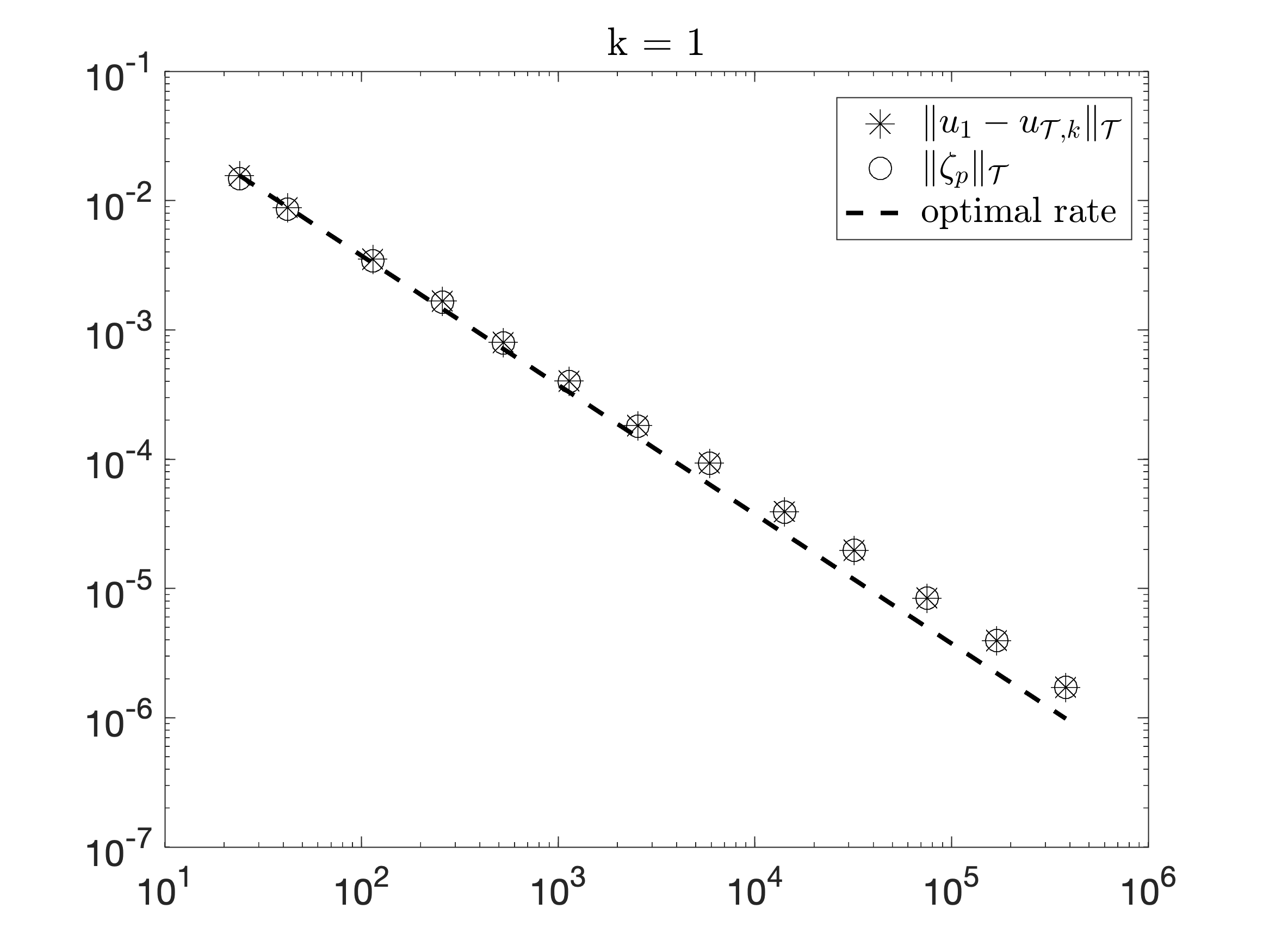}\\
	\includegraphics[width=.49\textwidth]{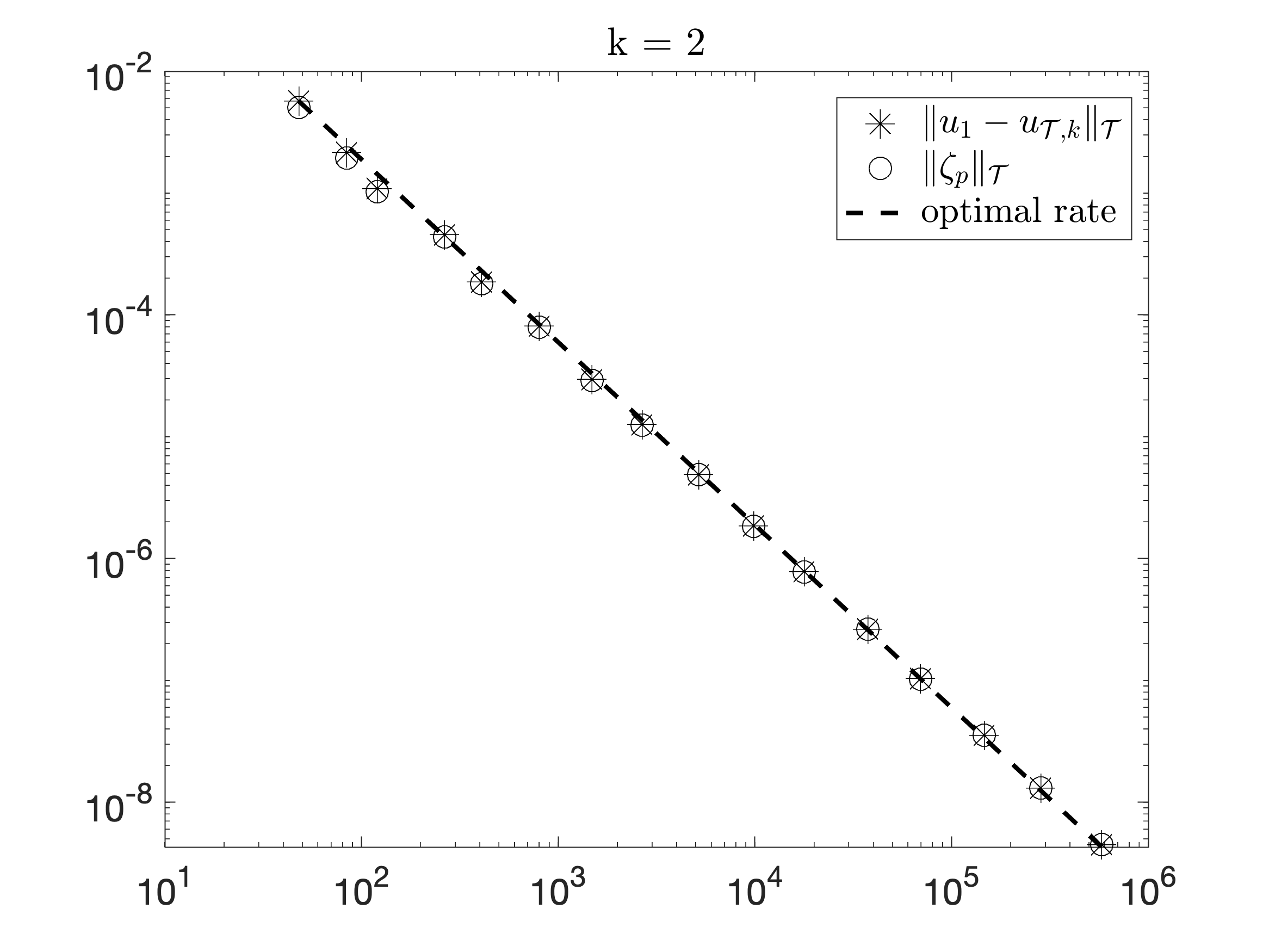}
    \includegraphics[width=.49\textwidth]{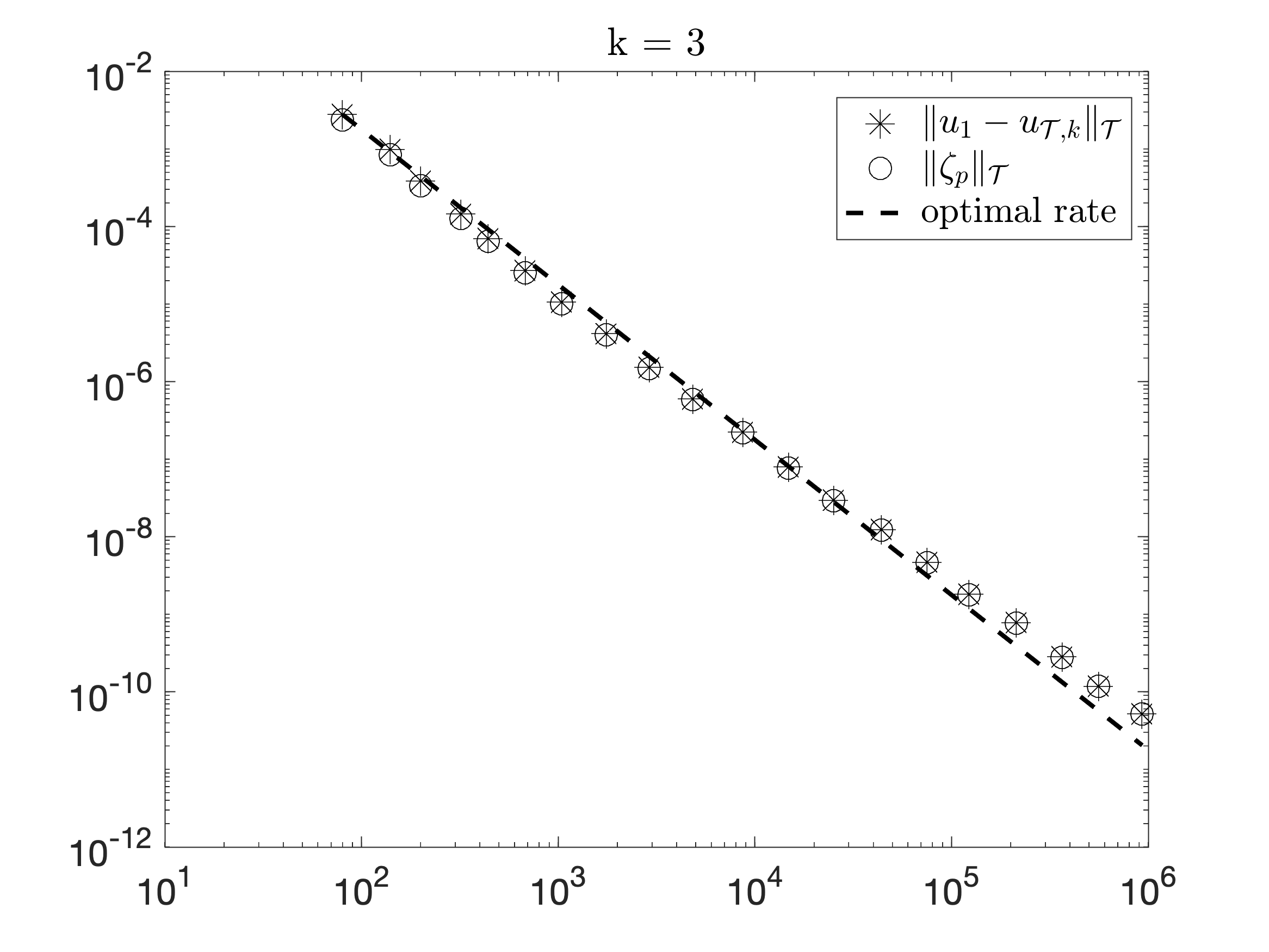}
	\caption{Test case \cref{eq: Test case 6} with singularity in $(0,0)$. Broken $H^1$ norm of the approximation error and of the $p$-estimator plotted against the theoretical optimal rate, for different values of the starting polynomial degree $k=0,1,2,3$, in a double logarithmic scale. \label{fig: TC6-H1 norm p-estimator}}
\end{figure}

\begin{figure}\centering
	\includegraphics[width=.49\textwidth]{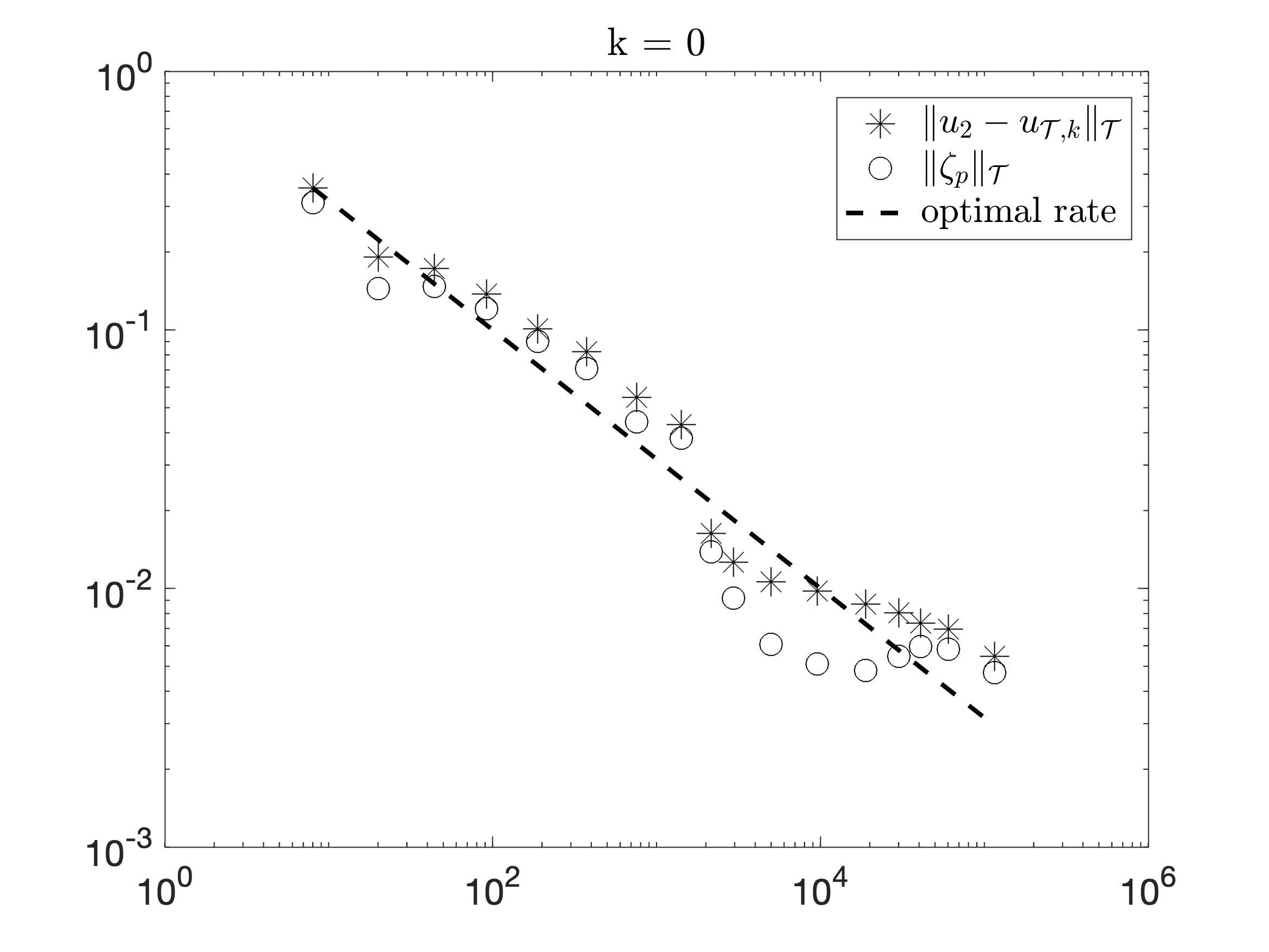}
	\includegraphics[width=.49\textwidth]{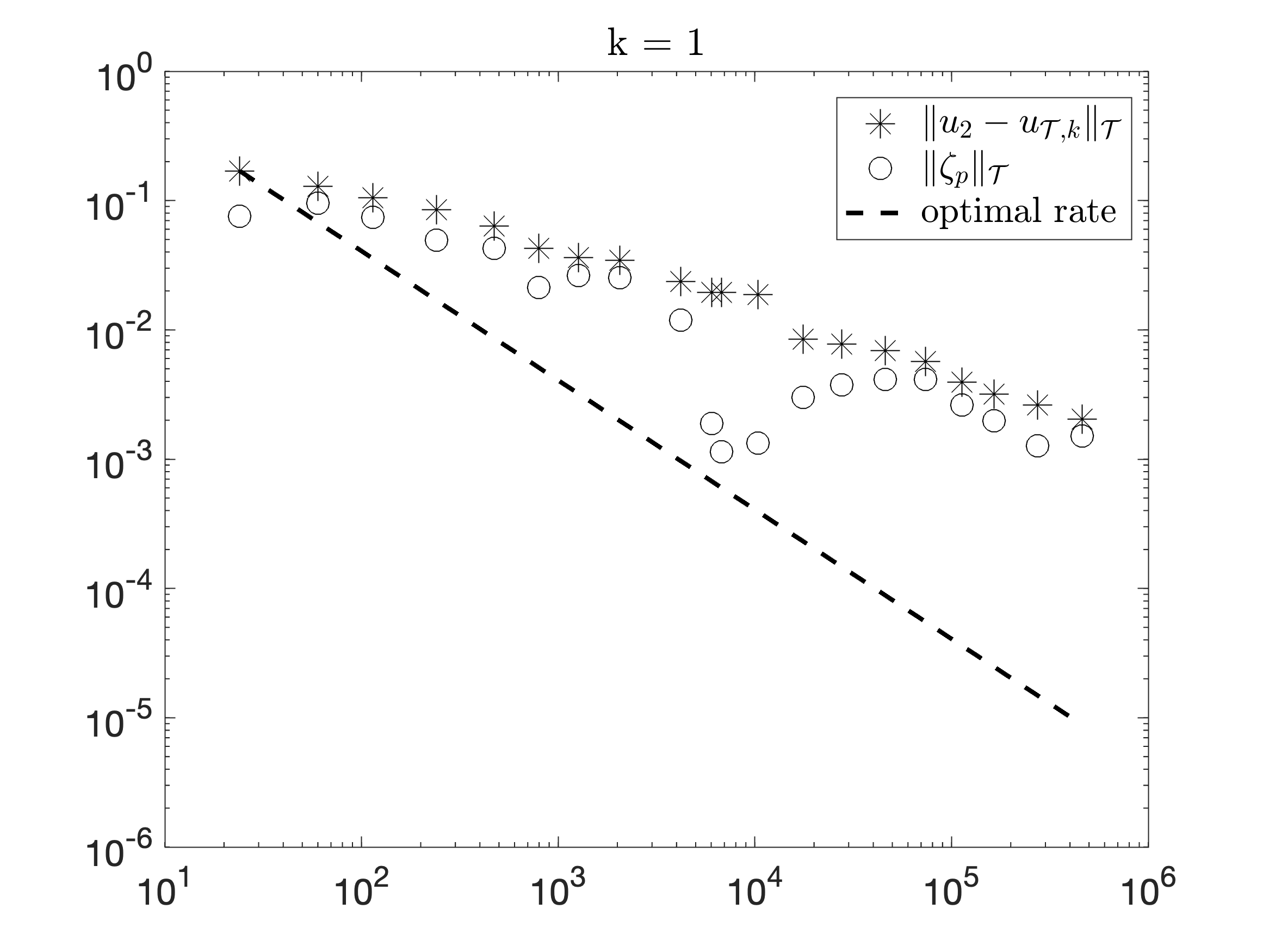}\\
	\includegraphics[width=.49\textwidth]{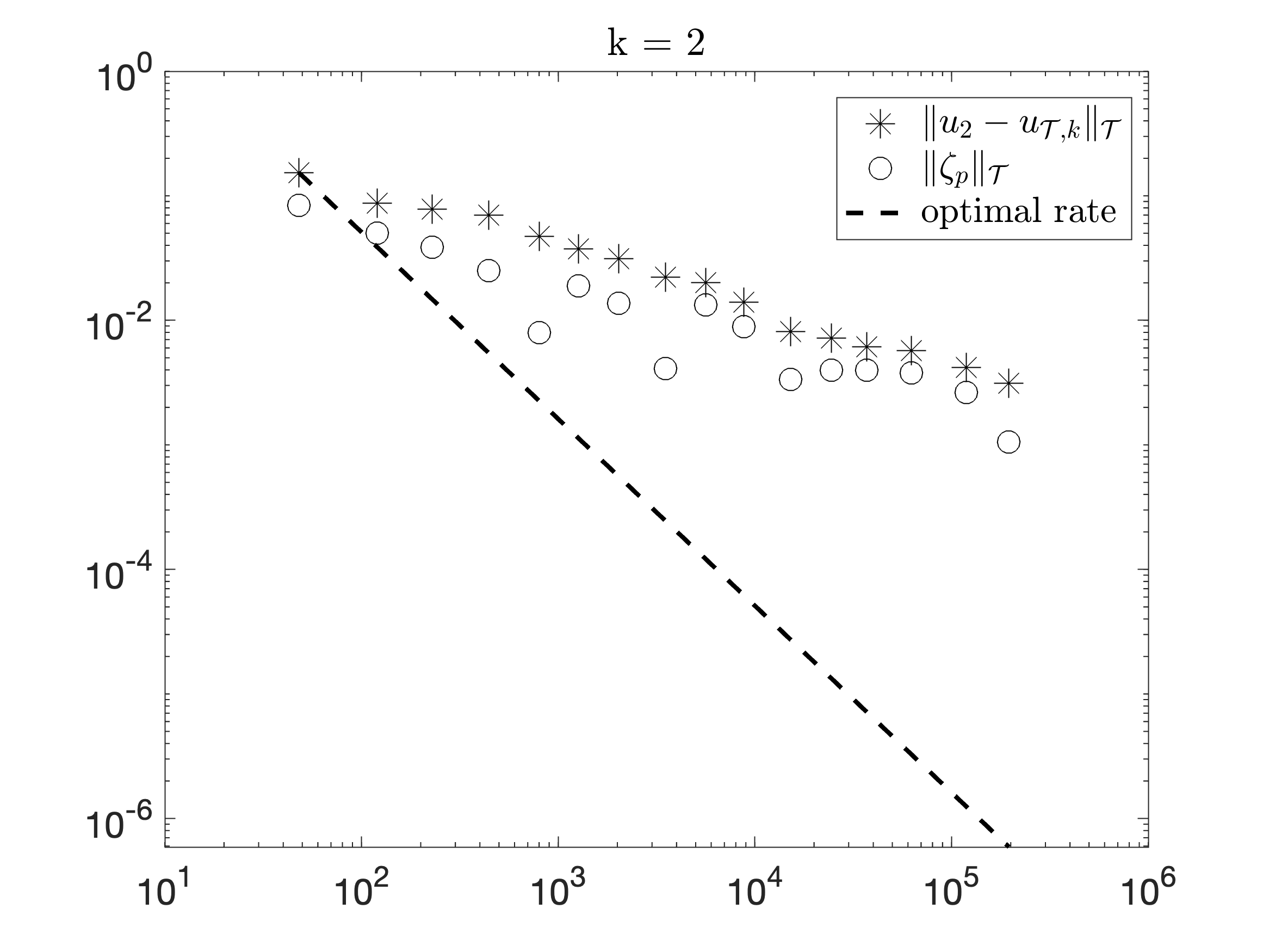}
    \includegraphics[width=.49\textwidth]{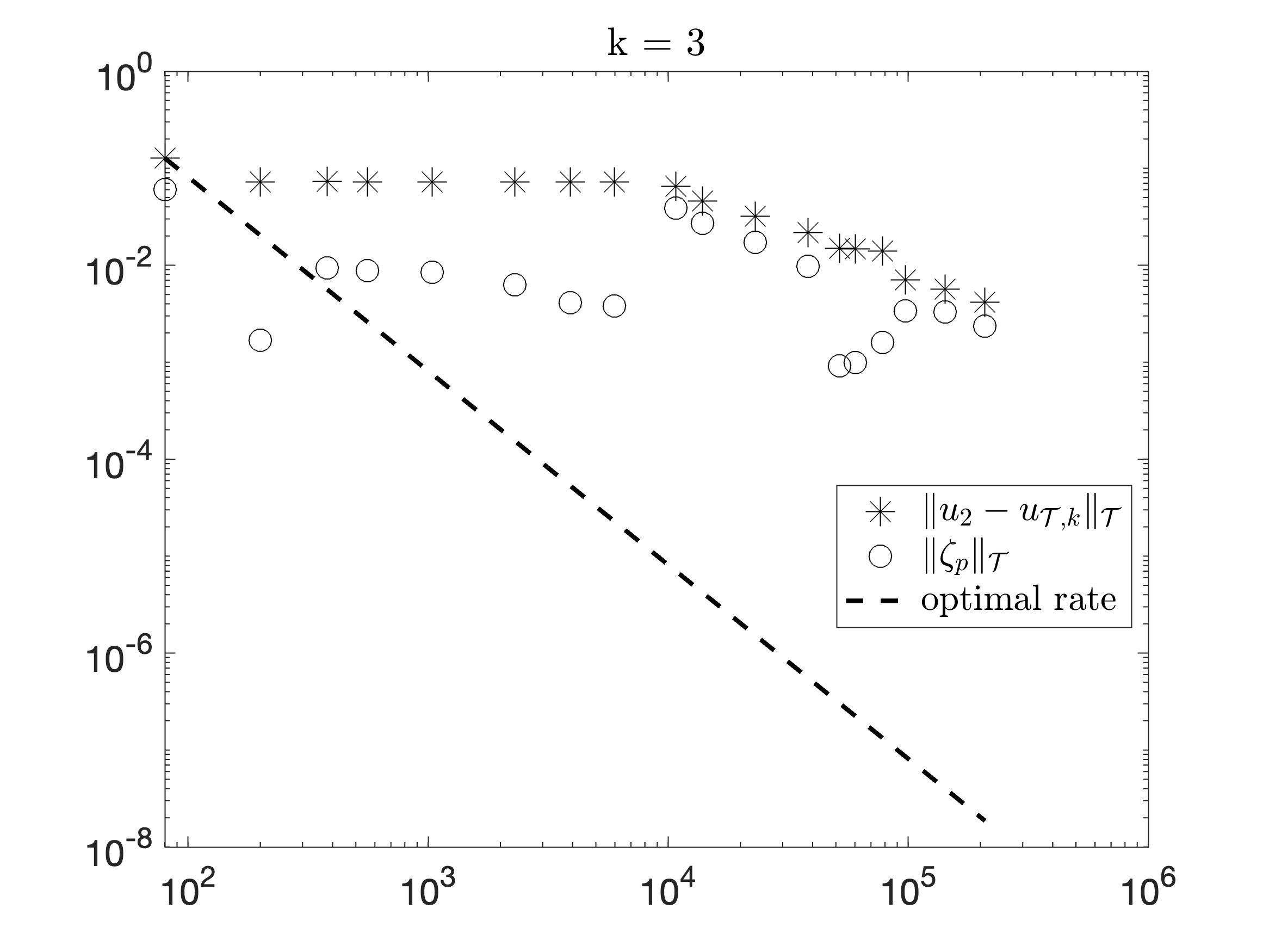}
	\caption{Test case \cref{eq: Test case 3} with discontinuity along the line $\mu=1/\sqrt{2}$. Broken $H^1$ norm of the approximation error and of the $p$-estimator plotted against the theoretical optimal rate, for different values of the starting polynomial degree $k=0,1,2,3$, in a double logarithmic scale. \label{fig: TC3-H1 norm p-estimator}}
\end{figure}

\subsection{Hierarchical \textit{h}-error estimator}
\label{sec: Hierarchical h error estimator}
Using once again the test cases \cref{eq: Test case 3} and \cref{eq: Test case 6}, we now keep $k^{\prime} \coloneqq k$ fixed and we construct $\T^{\p}$ by uniform refinement of $\T$, i.e., every element in $\T$ is subdivided in $4$ new elements with halved edge length. The error estimator is now
\begin{align}
    \label{eq:hierarchical_error_estimator2}
	\zeta_h \coloneqq u_{\T^{\p}, k} - u_{\T,k}.
\end{align}
Some comments are due for the computation of $u_{\T^{\p}, k}$, for which we use, as in definition \eqref{eq: SIP bilinear form} but with $\T$ replaced by $\T^{\p}$, the bilinear form $\aad:V_{\T^{\p},k}\times V_{\T^{\p},k}\to \mathbb{R}$. Since $V_{\T,k}$ is a subspace of $V_{\T^{\p},k}$, we require, similar to \cite{Karakashian_2003}, that $a_h$ is the restriction of $\aad$ to $V_{\T,k}$, in the sense that
\begin{equation}
    \label{eq: a_h as restriction of a_had}
    a_h(v,w) = \aad(v,w) \qquad \forall u,v\in V_{\T,k}.
\end{equation}
Comparing the penalty terms in $a_h$ and $\aad$ shows that the above restriction is fulfilled if $\alpha_F = 2\alpha^{\p}_F$. Since we need to ensure discrete stability of both bilinear forms, we choose $\alpha_F^{\p} \coloneqq 1/2 + C_{dt}(k)$.

\Cref{fig: TC6-H1 norm h-estimator} and \Cref{fig: TC3-H1 norm h-estimator} show the optimality of the estimator for the manufactured solution $u_1$ with point singularity defined in \cref{eq: Test case 6}, and sub-optimality in the case of discontinuous exact solutions \cref{eq: Test case 3}, except for $k=0$, where a similar comment as for the $p$-hierarchical estimator applies.
We note that in all cases, the estimator is close to the actual error. 
Moreover, as shown in Figure~\ref{fig: Comparison} we observe that the estimator is able to detect the line discontinuity present in $u_2$.
We note that the condition in \cref{eq: a_h as restriction of a_had} is not essential for the results shown in this section. In fact, similar results can be obtained by using $\alpha_F=\alpha_{F^{\p}}=1/2+C_{dt}(k)$. The condition \cref{eq: a_h as restriction of a_had} is required in the next section.

\begin{figure}\centering
	\includegraphics[width=.49\textwidth]{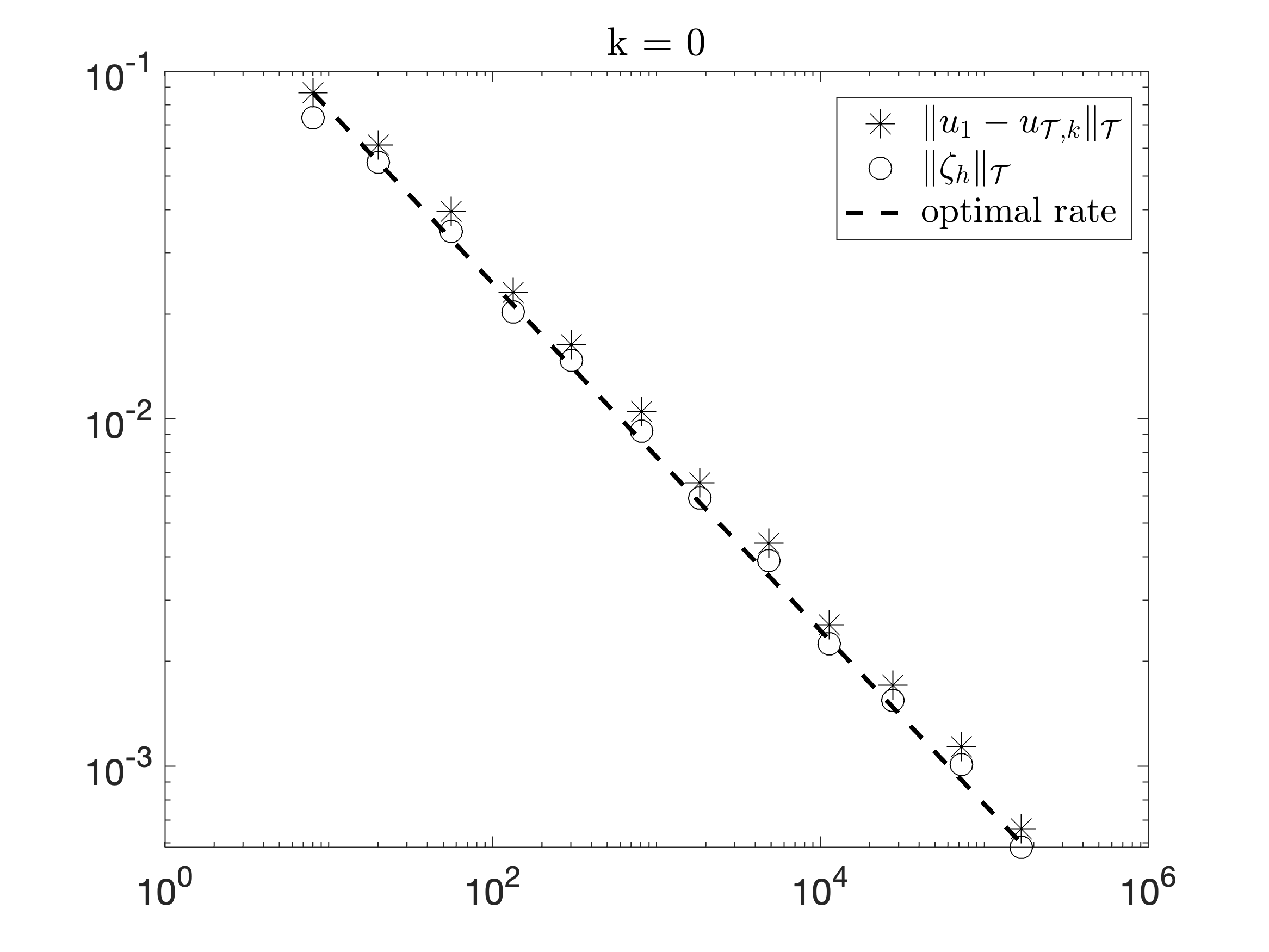}
	\includegraphics[width=.49\textwidth]{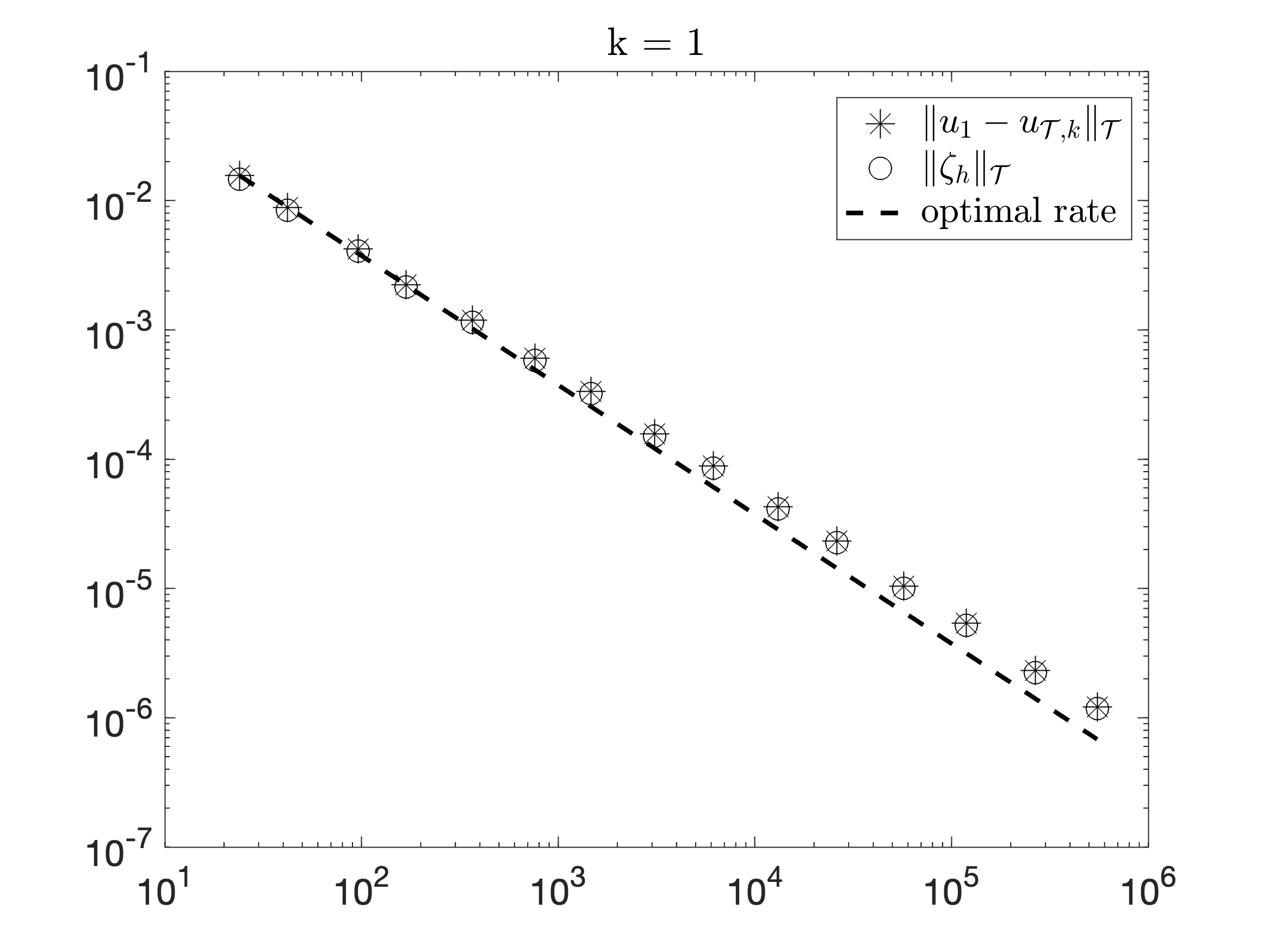}\\
	\includegraphics[width=.49\textwidth]{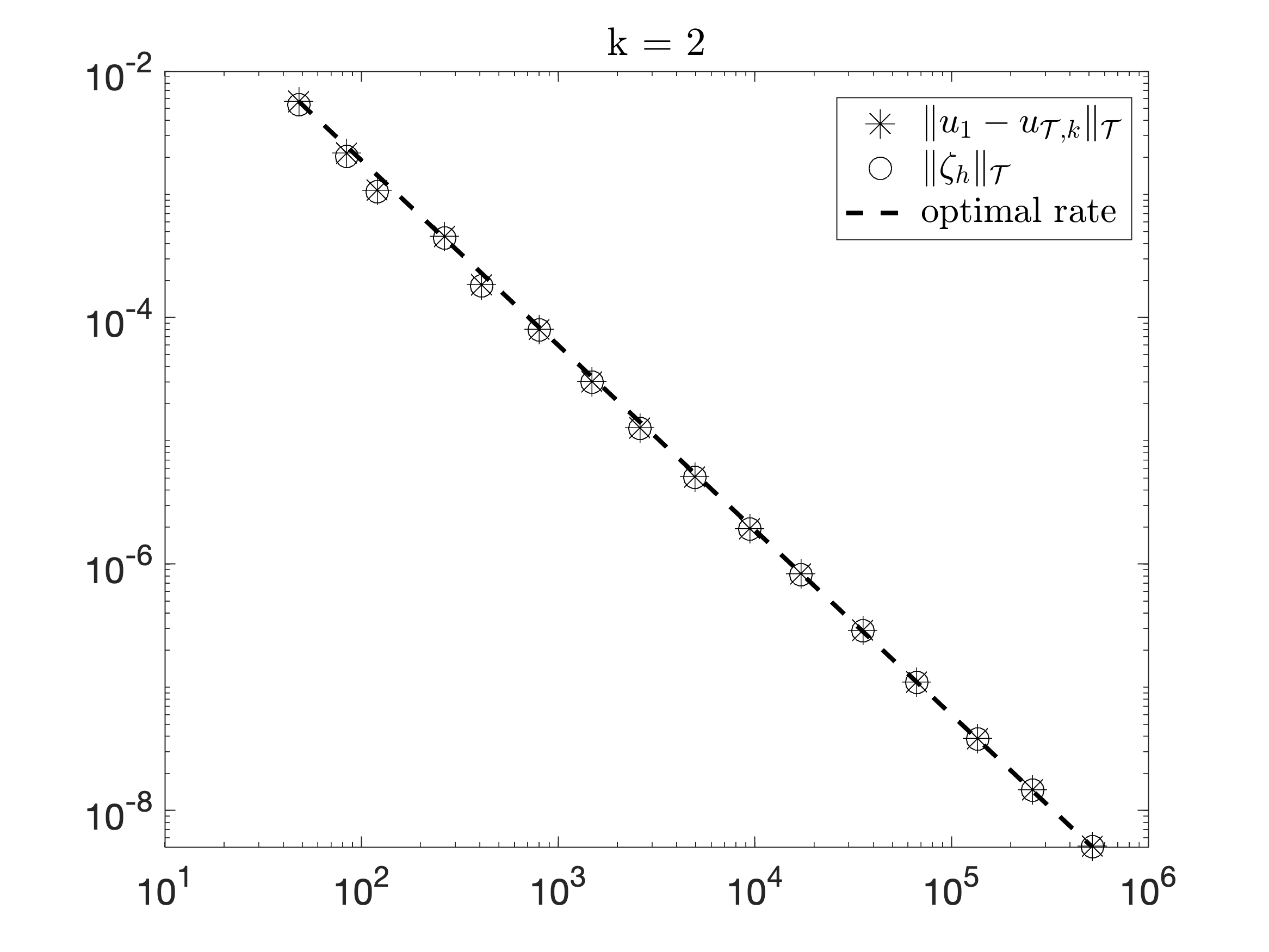}
    \includegraphics[width=.49\textwidth]{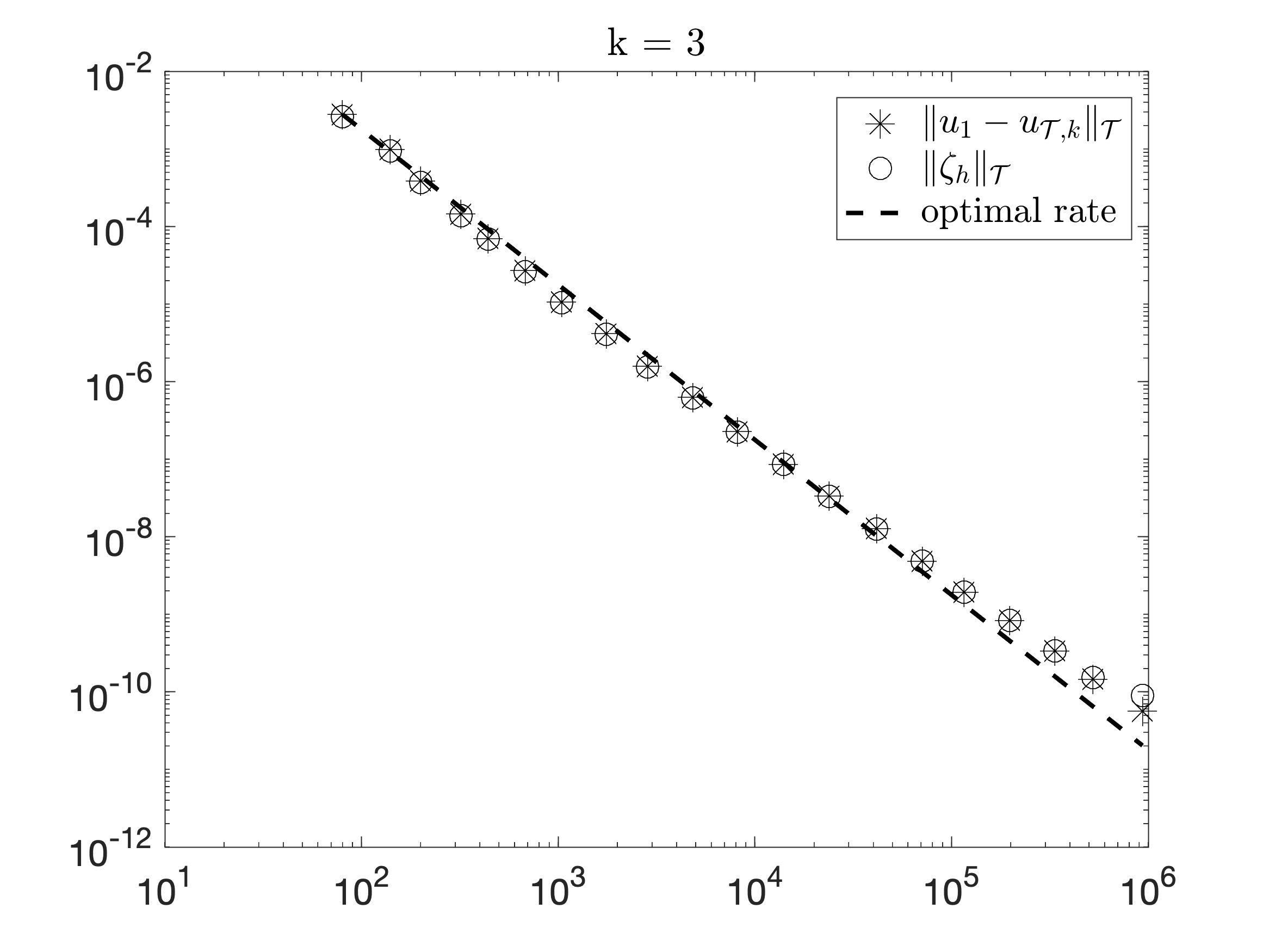}
	\caption{Test case \cref{eq: Test case 6} with singularity in $(0,0)$. Broken $H^1$ norm of the approximation error and of the $h$-estimator plotted against the theoretical optimal rate, for different values of the  polynomial degree $k=0,1,2,3$, in a double logarithmic scale. \label{fig: TC6-H1 norm h-estimator}}
\end{figure}

\begin{figure}\centering
	\includegraphics[width=.49\textwidth]{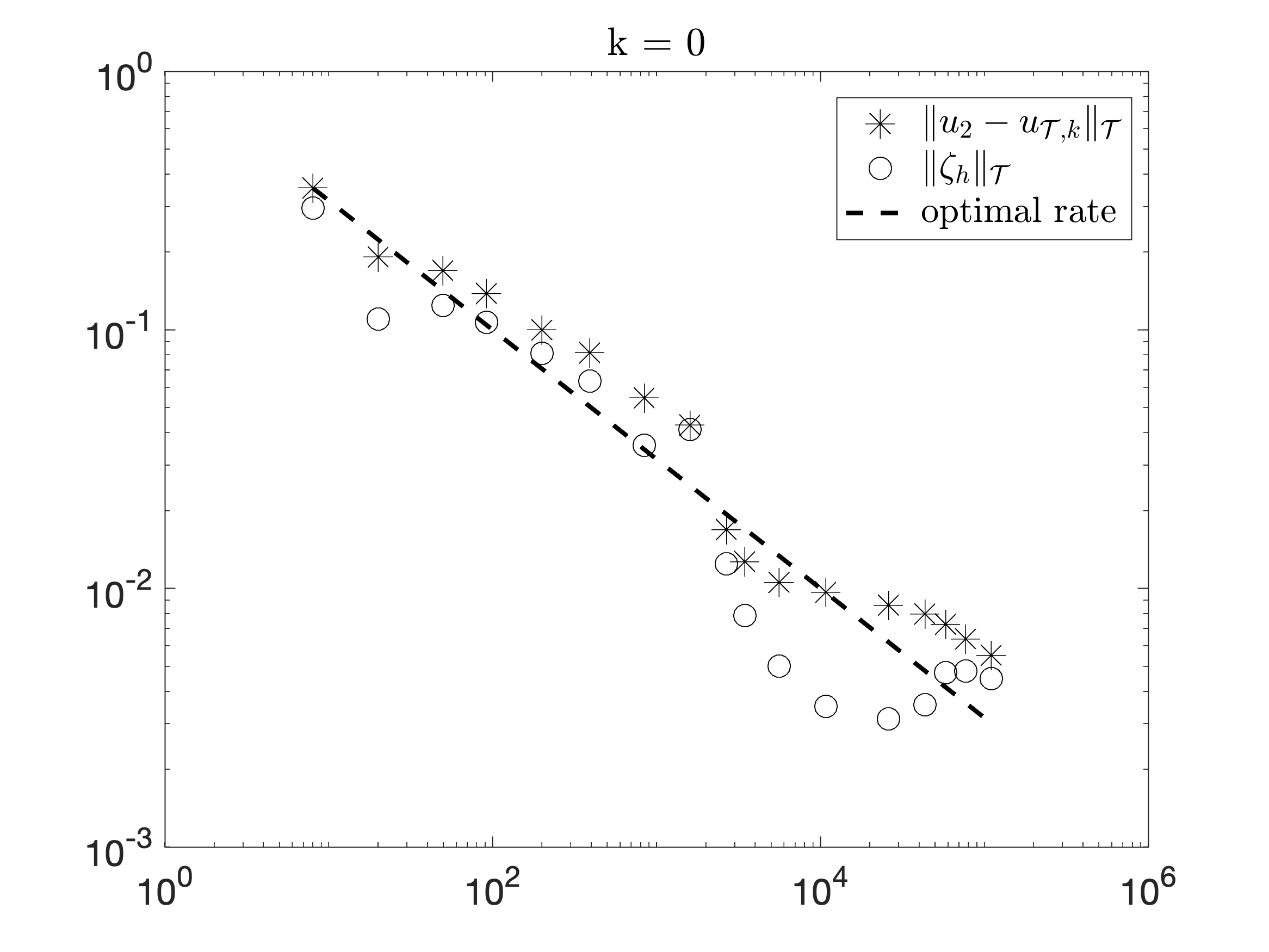}
	\includegraphics[width=.49\textwidth]{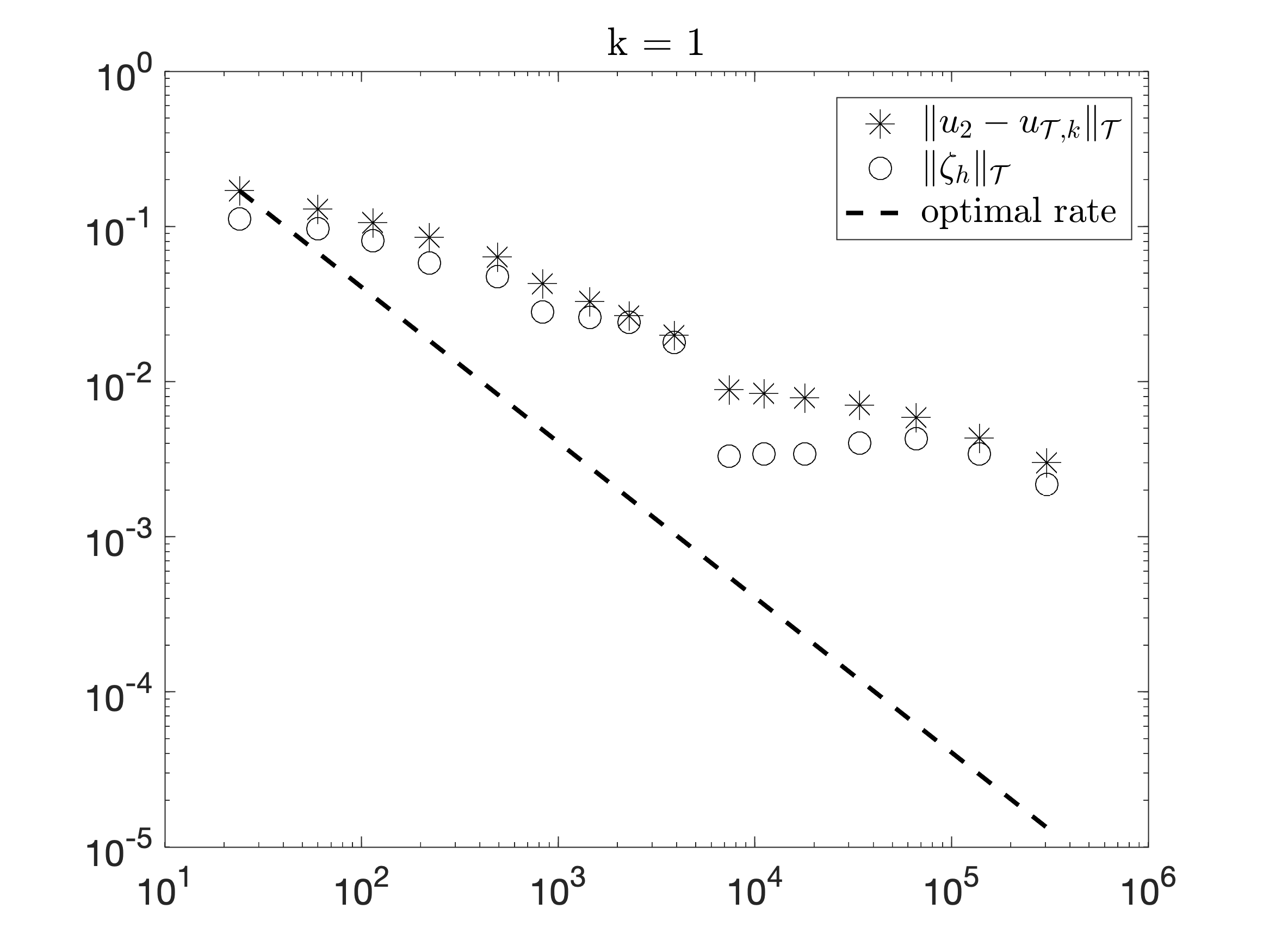}\\
	\includegraphics[width=.49\textwidth]{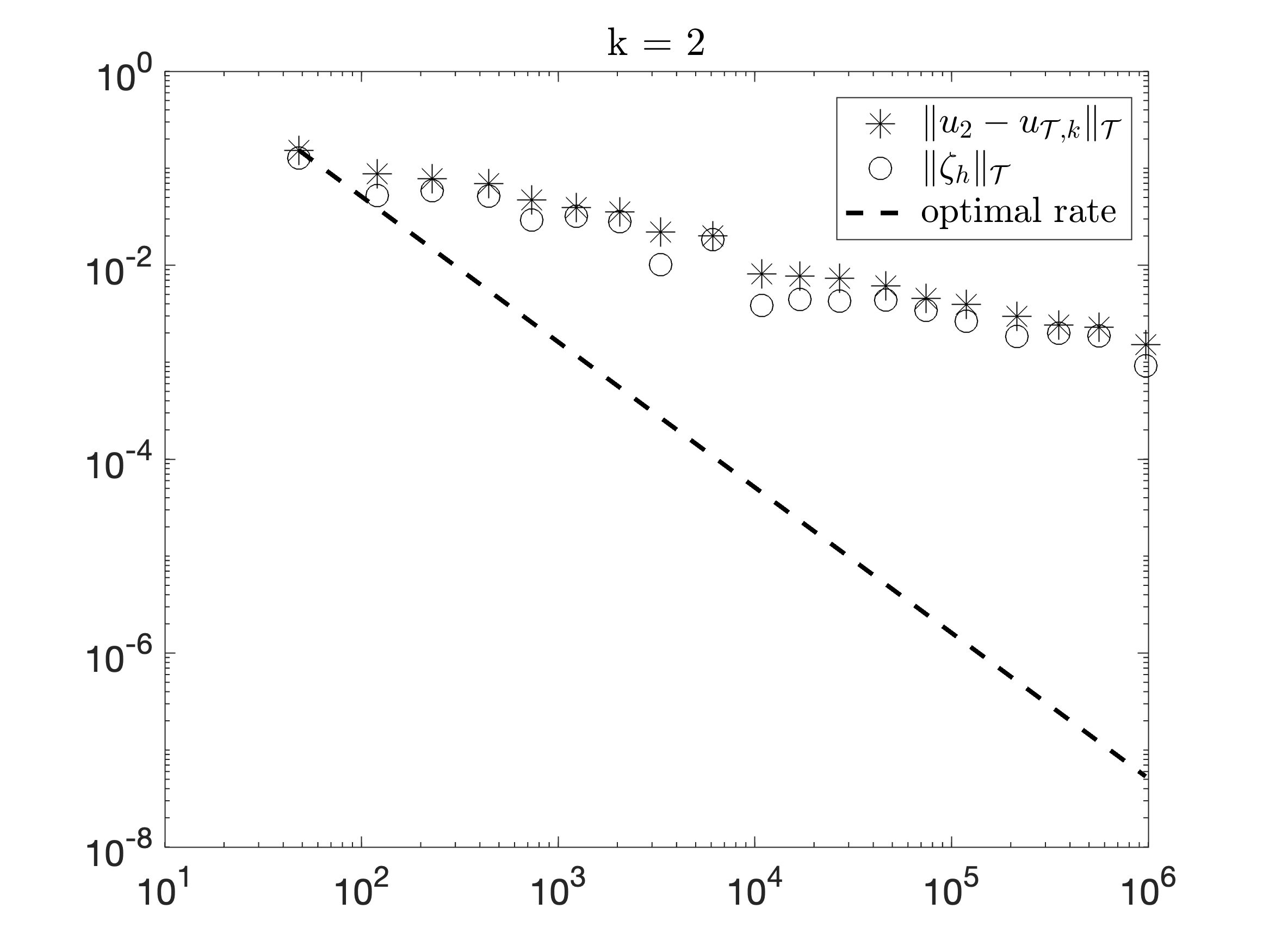}
    \includegraphics[width=.49\textwidth]{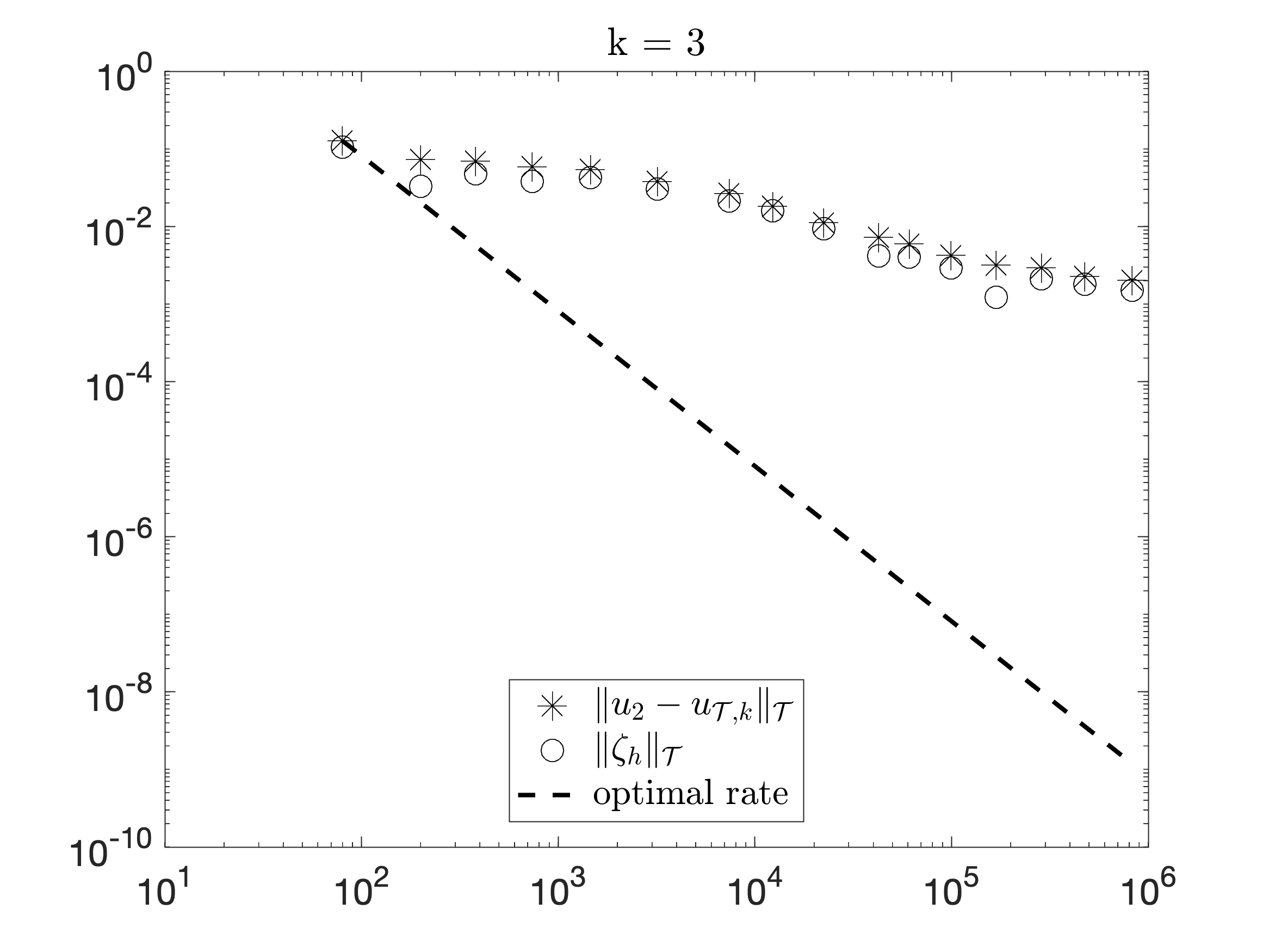}
	\caption{Test case \cref{eq: Test case 3} with line discontinuity. Broken $H^1$ norm of the approximation error and of the $h$-estimator plotted against the theoretical optimal rate, for different values of the polynomial degree $k=0,1,2,3$, in a double logarithmic scale. \label{fig: TC3-H1 norm h-estimator}}
\end{figure}

\begin{figure}\centering
	\includegraphics[width=.49\textwidth]{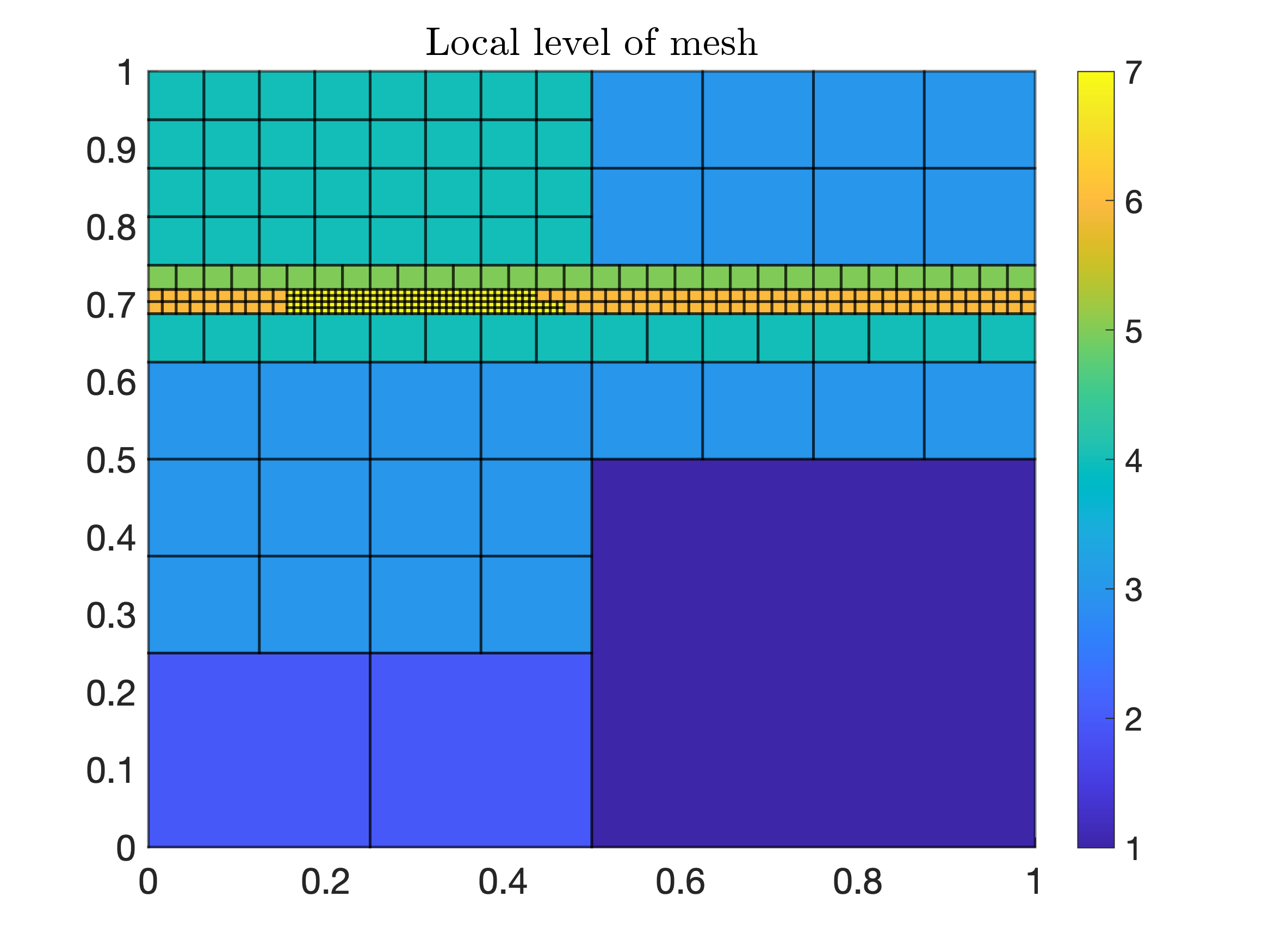}
	\includegraphics[width=.49\textwidth]{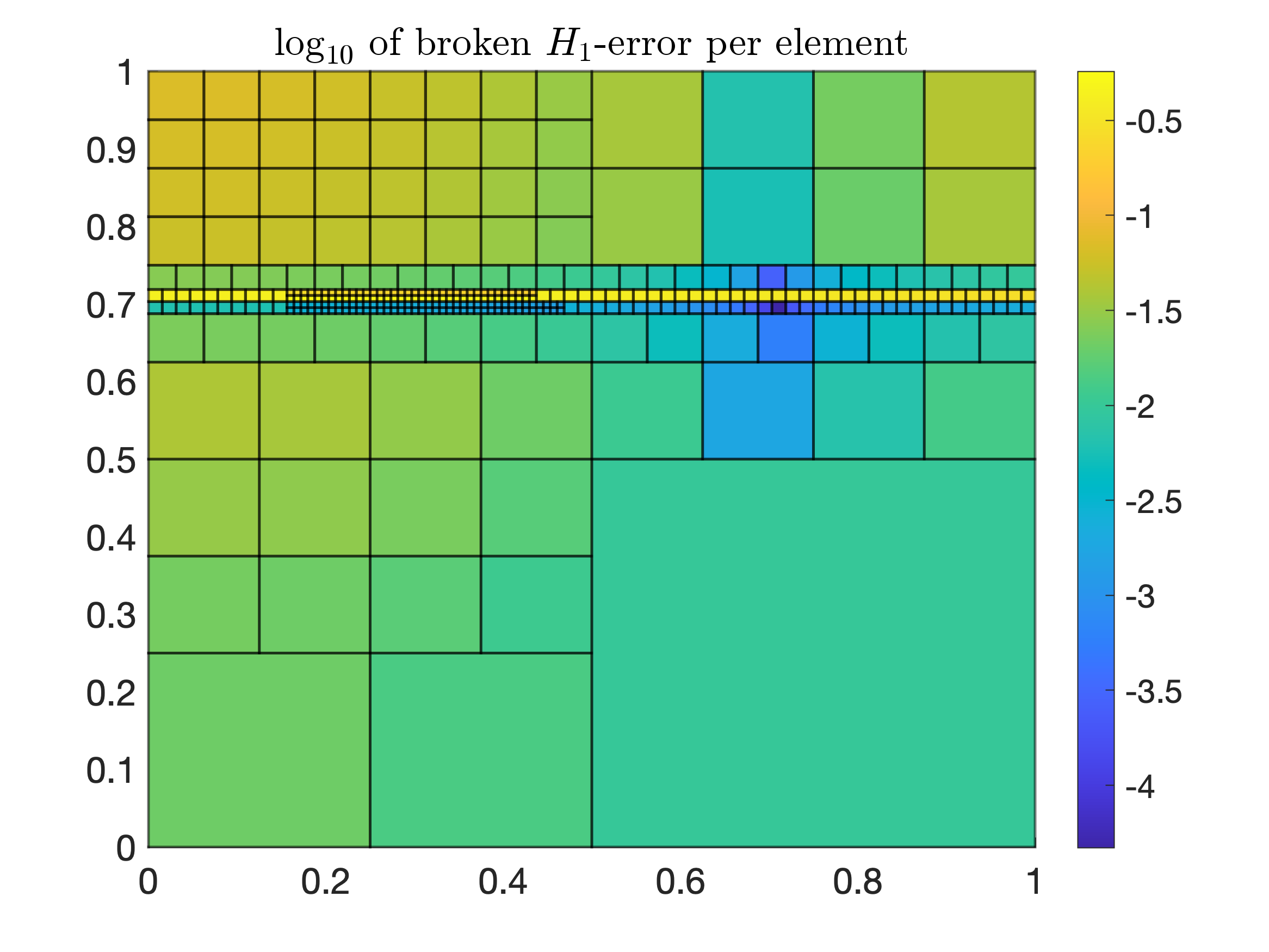}	
	\caption{Non-smooth test case \cref{eq: Test case 3}. Left: Locally refined mesh with local mesh sizes varying from $1/2$ to $1/2^7$ for $N = 349$ elements obtained using the error indicator $\zeta_h$ defined in \cref{eq:hierarchical_error_estimator2}. Right: Broken $H^1$-error for the grid shown left.\label{fig: Comparison}}
\end{figure}

\subsection{Error estimator based on the solution of local problems}
\label{sec: Error estimator based on the solution of local problems}

Since the computation of the global error estimators $\zeta$ presented in \Cref{sec: Hierarchical p error estimator} and \Cref{sec: Hierarchical h error estimator} is in general expensive, we investigate also an error estimator based on the solution of local problems, as presented in \cite{Karakashian_2001, Karakashian_2003} for corresponding elliptic problems.
In this approach, the computed solution $u_{\T, k}$ is understood as the coarse-mesh approximation to some function, here $u_{\T^{\p},k}$. Instead of using $\zeta=u_{\T,k}-u_{\T^\p,k}$ as before, local approximations $\eta_K$ are computed element-wise.

For each $K\in \T$, we consider the local space $V_{\T^{\p},k}(K)$ obtained by restricting functions in $V_{\T^{\p},k}$ to $K$. By extending functions in $V_{\T^{\p},k}(K)$ to zero outside of $K$, $V_{\T^{\p},k}(K)$ becomes a subspace of $V_{\T^{\p},k}$. Indeed,
\begin{equation}
    \label{eq: Vad as direct sum}
    V_{\T^{\p},k} = \bigoplus_{K\in \T} V_{\T^{\p},k}(K),
\end{equation}
where $\oplus_{K\in\T}$ denotes the direct sum of subspaces. On $V_{\T^{\p},k}(K) \times V_{\T^{\p},k}(K)$ we introduce the (local) bilinear form $\aad^K$ as the restriction of $\aad$ to $\Vad(K) \times \Vad(K)$. This bilinear form inherits continuity and coercivity from $\aad$. Using \Cref{lem:discrete_stability} there holds
\begin{equation}
    \label{eq: coercivity of aad^K}
    \aad^K(v,v) \geq \frac{1}{2}\norm{v}^2_{\Vad,K},
\end{equation}
where $\normc_{\Vad,K}$ is the restriction of $\normc_{\Vad}$ to $\Vad(K)$. Here, $\normc_{\Vad}$ is defined according to \cref{eq: norm Vh} as
\begin{equation}
    \|v\|_{\Vad}^2  = a_{h^{\p}}^e(v,v) + \sum_{F\in \mathcal F^{vi}_{\had}} \frac{H^{\p}_F}{C_{dt}(\kad)} \norm{\avg{\fluxh{v}}}_{L^2(F;\mu)}^2,\quad \forall v\in \Vad.
\end{equation}
Let $\uad$ be the discontinuous Galerkin approximation of $u$ on $\Vad$, i.e. 
\begin{equation}
    \label{eq: Galerkin approximation on Vad}
    \aad(\uad, v) = \ip{f}{v} + \langle g, v \rangle \qquad \forall v\in \Vad.
\end{equation}
At this point we observe that \eqref{eq:DG}, \eqref{eq: a_h as restriction of a_had} and \eqref{eq: Galerkin approximation on Vad} imply, for all $v\in V_h$,
\begin{equation}
    \label{eq: orthogonality of zeta to functions in Vh}
    \begin{aligned}
    \aad(\uad-u_{\T,k}, v) & = \aad(\uad, v) - \aad(u_{\T, k},v) \\ 
                      & = \aad(\uad, v) - a_h(u_{\T, k},v) \\
                      & = \ip{f}{v} + \langle g, v \rangle - \ip{f}{v} - \langle g, v \rangle = 0.
\end{aligned}
\end{equation}
Eventually we introduce the functions $\{\eta_K \in \Vad(K) | K\in\T_h\}$ as solutions to the local problems 
\begin{equation}
    \label{eq: local problem}
    \aad^K(\eta_K, v) = \aad(u_{\T^{\p}, k}-u_{\T, k}, v) = \ip{f}{v} + \langle g, v \rangle - \aad(u_{\T, k},v) \quad \forall v \in \Vad(K).
\end{equation}
Each $\eta_K$ can be computed independently of each other. The function $\eta = \sum_{K\in \T} \eta_K \in \Vad$ then may serve as an approximation to the estimator $\zeta = u_{\T^{\p},k}-u_{\T,k}$ on $\Vad$, as in \cite{Karakashian_2003} for elliptic problems. Following \cite[Theorem 4.1]{Karakashian_2003}, we can prove a lower bound for $\zeta$ in terms of the local error estimator $\eta$, i.e.
\begin{lemma}\label{lem:local_est}
    We have that
    \begin{equation}
        \label{eq: Equivalence between zeta and eta}
         \norm{\eta}_{\Vad} \leq 2(C_{dt}(k)+\alpha_F^{\p})\norm{\zeta}_{\Vad}.
    \end{equation}
\end{lemma}
\begin{proof}
    We first rewrite \eqref{eq: local problem} in terms of the estimator:
    \begin{equation}
        \label{eq: Local problems in terms of zeta}
        \aad^K(\eta_K, v) = \aad(\zeta, v) \quad \forall v\in \Vad(K).
    \end{equation}
    Plugging $v = \eta_K \in \Vad(K)$ into the previous equation and recalling that $\eta = \sum_K \eta_K$, we have
    \begin{equation}
        \label{eq: boundedness of aad_K{zeta,eta}}
        \sum_{K\in \T} \aad^K(\eta_K, \eta_K) = \aad(\zeta,\eta) \leq (C_{dt}(\kad)+\alpha_F^{\p}) \norm{\zeta}_{\Vad}\norm{\eta}_{\Vad},
    \end{equation}
    where we used \Cref{cor:discrete_boundedness} in the last step.
    Coercivity of $\aad^K$, expressed in \eqref{eq: coercivity of aad^K}, and \cref{eq: Vad as direct sum} entail
    \begin{equation}
        \label{eq: application of coercivity of aad^K}
        \sum_{K\in \T} \aad^K(\eta_K, \eta_K) \geq \dfrac{1}{2} \sum_{K\in\T}\norm{\eta_K}^2_{\Vad,K} \geq \dfrac{1}{2} \norm{\eta}^2_{\Vad}.
    \end{equation}
    Combining \eqref{eq: boundedness of aad_K{zeta,eta}} and \eqref{eq: application of coercivity of aad^K} we have eventually
    \begin{equation}
        \dfrac{1}{2} \norm{\eta}^2_{\Vad} \leq \sum_{K\in \T} \aad^K(\eta_K, \eta_K) \leq (C_{dt}(\kad)+\alpha_F^{\p}) \norm{\zeta}_{\Vad}\norm{\eta}_{\Vad},
    \end{equation}
    which concludes the proof.
 \end{proof}   
Due to the lack of appropriate interpolation operators for functions in the space $V$, one cannot adapt the proofs given in \cite{Karakashian_2003} in a straight-forward fashion to show a bound of $\zeta$ in terms of the local contributions $\eta$. In fact, some preliminary numerical tests, based on the broken $H^1$-norm \cref{eq: broken H1 norm}, suggest that the desired equivalence of $\zeta$ and $\eta$ might not be true.
In a similar spirit, our preliminary numerical tests indicate that standard residual error estimators are either not reliable or efficient, which, again, may be explained by the lack of suitable interpolation error estimates required to obtain the correct scaling in terms of the local mesh size of the different local contributions, cf., \cite[Section~5.6]{diPietroErn} or \cite{AinsworthOden2000,Verfuerth2013}.
Therefore, we investigate in the next section another error estimator based on local averages.

\subsection{Error estimator based on averaging the approximate solution}

In the context of a posteriori error estimation and adaptive mesh refinement, ZZ-error estimators named after Zienkiewicz
and Zhu \cite{ZienkiewiczZhu1992} are widely used in practice. Compared to the previously mentioned hierarchical error estimators, their major advantage is the fact that no further mesh nor a further solution is required. 
We consider the case $\kz=\kmu=0$.
In order to obtain a reliable error estimator, one simply takes a discontinuous $u_h \in V_h$ and approximates it by some continuous piecewise linear polynomial $\tilde u_h$ by a post-processing step.
In the presence of a geometrically conforming triangulation, such a continuous piecewise polynomials $\tilde u_h$ can be described as a linear combination of the well-known Lagrange nodal basis functions. However, our approximation involves hanging nodes and we therefore restrict the construction to the set of regular nodes $\mathcal N_h$, i.e. 
\begin{equation*}
\mathcal N_h =
\{
\nu \text{ node in } \mathcal T_h \ : \
\nu \in K \text{ implies }\nu \text{ vertex of }K
\ \forall K \in \mathcal T_h 
\} \ .
\end{equation*}
If a regular node $\nu \in \mathcal N_h$ is shared by four quadrilaterals $K_1,...,K_4$ of the same area, the idea is to set the value of a continuous polynomial to 
$\frac 1 4 \left( 
{{u_h}_|}_{K_1} (\nu) 
+{{u_h}_|}_{K_2} (\nu) 
+{{u_h}_|}_{K_3} (\nu) 
+{{u_h}_|}_{K_4} (\nu) \right)$ at the node $\nu$. Taking into account the possibility of quadrilaterals of different area, for a node $\nu \in \mathcal N_h$, we define by $\omega_\nu$ the union of all elements of $K \in \mathcal T$ sharing the vertex $\nu$.
The continuous piecewise linear averaging $\tilde u_h$ is the defined such that
    \begin{equation}
\tilde u_h(\nu)=
\sum\limits_{K \in \mathcal{T}, K\subset\omega_\nu}
 \frac{|K|}{|\omega_\nu|} {{u_h}_|}_K (\nu) 
    \end{equation}
holds for each regular node $\nu \in \mathcal N_h$.
The averaging error estimator is then defined by 
\begin{equation}
\eta_A^2 := \sum_{K\in\T} \eta_{A,K}^2,\qquad\text{with } \eta_{A,K}:=\| u_h - \tilde u_h \|_{L^2(K)},
\end{equation}
where the local contributions are used to refine the mesh using D\"orfler marking as described above.
\Cref{fig:AverageConvergencerates} shows the convergence rates for adaptively refined meshes using the averaging indicator for the test \cref{eq: Test case 3}.
The indicator behaves correctly and replicates the curve of the actual error. These curves have the same slope as the optimal rate $1/\sqrt{N}$ curve, with $N$ number of elements in the quad-tree mesh, also shown for comparison. 

In comparison to the hierarchical estimators, cf. \Cref{fig: TC3-H1 norm p-estimator} and \Cref{fig: TC3-H1 norm h-estimator}, the averaging error estimator follows the actual error curve more closely.

\begin{figure}\centering
	\includegraphics[width=.49\textwidth]{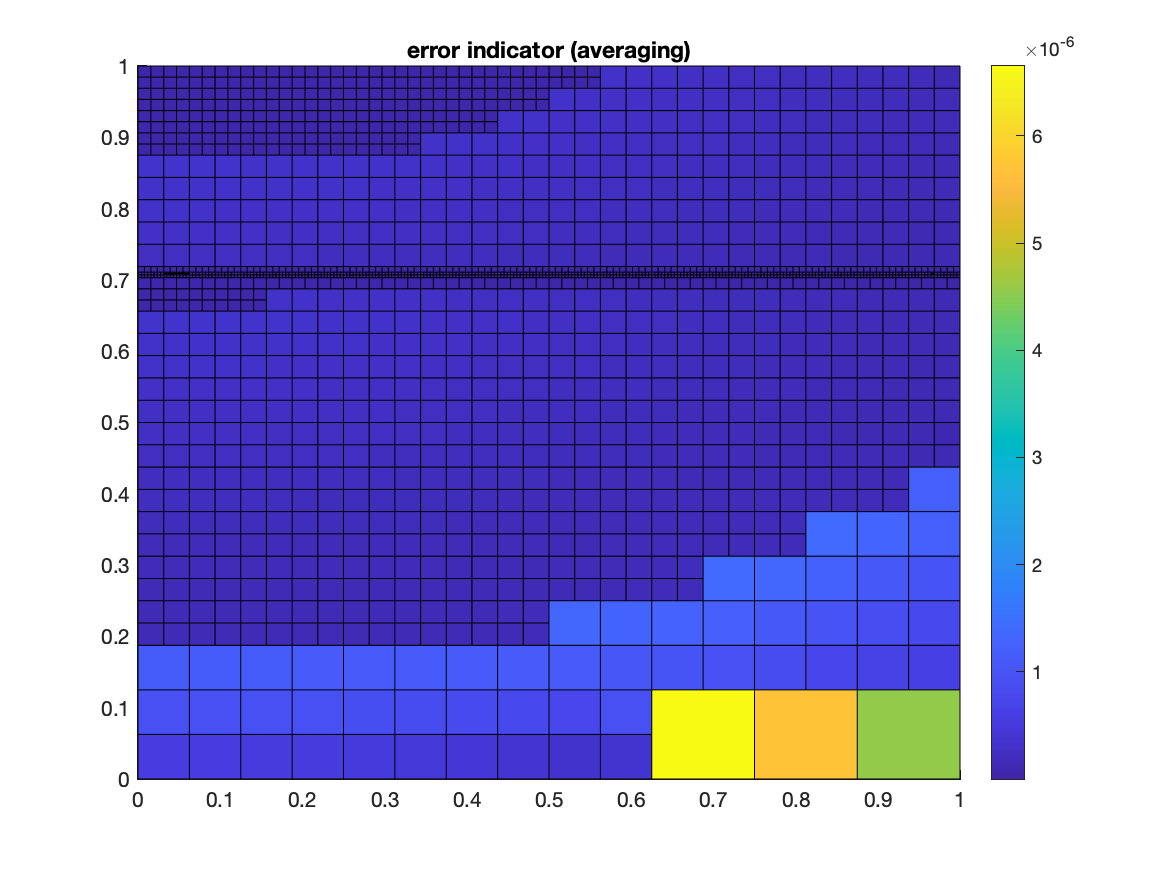}
	\includegraphics[width=.49\textwidth]{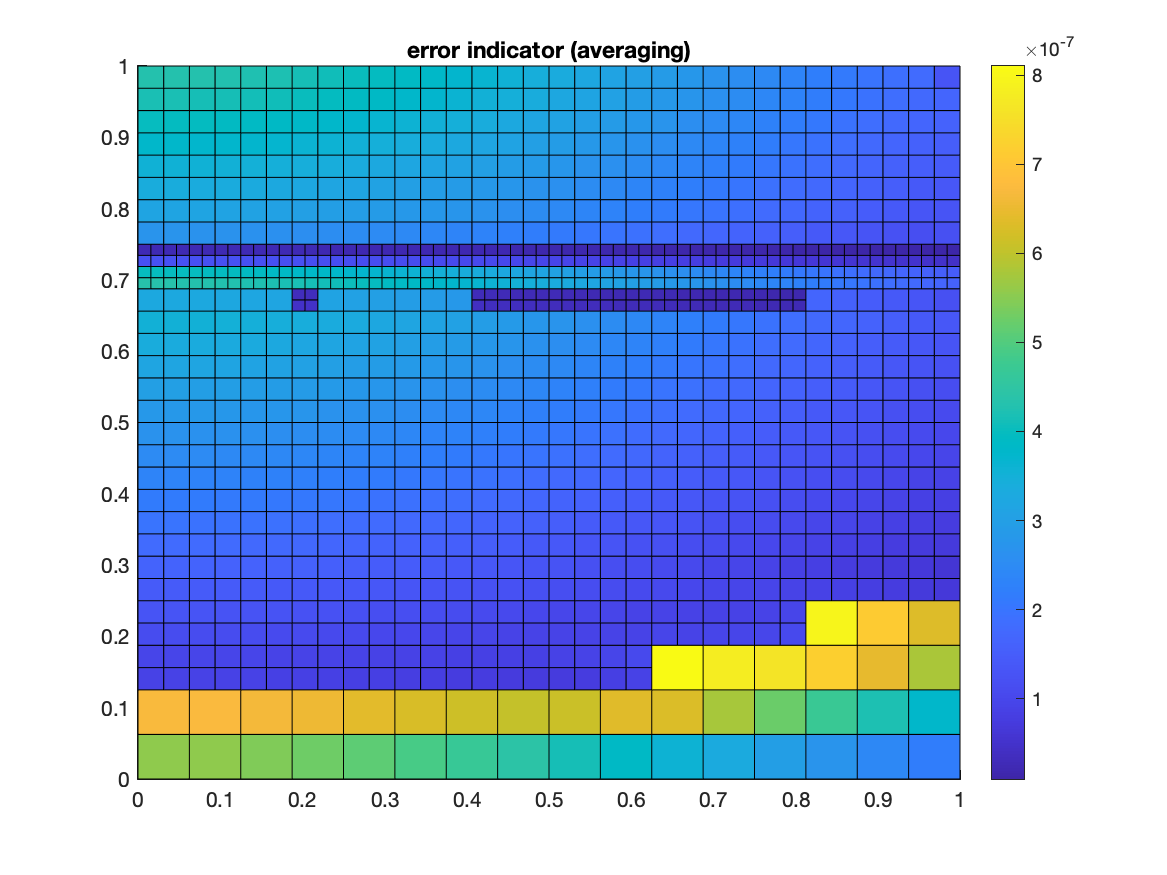}\\
	\includegraphics[width=.99\textwidth]{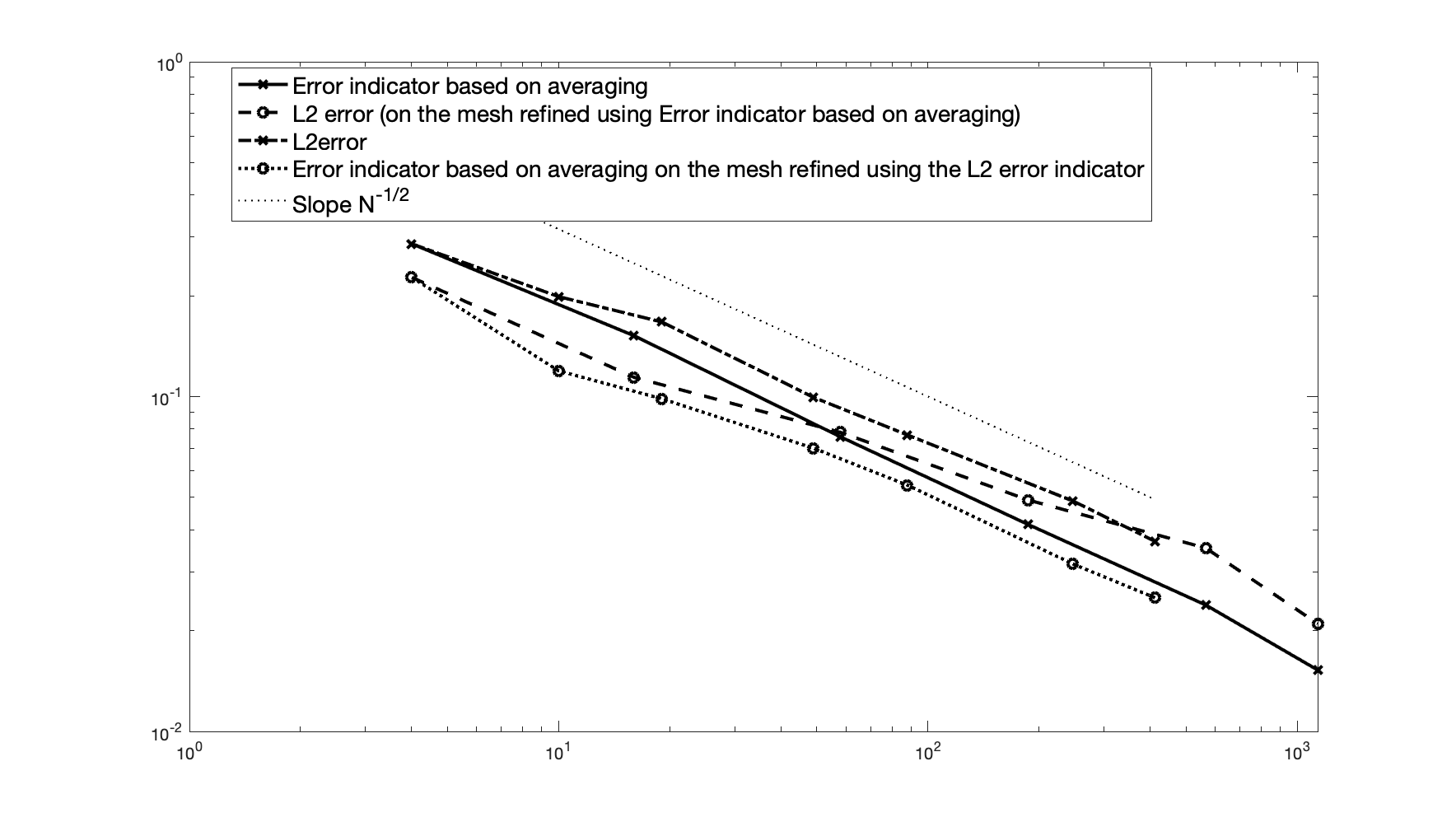}
	\caption{Non-smooth test case \cref{eq: Test case 3}. Top Locally refined mesh with the average error estimator after $6$ (left) and $9$ (right) refinements. Bottom: Convergence of the DG solution and the averaging error estimator on adaptively refined grids as well as the optimal rate $1/{\sqrt{N}}$ (light dotted line) in a double logarithmic scale. The dashed line with o shows the behavior of the $L^2$-error using the averaging error estimator for grid adaptation. The solid line with x shows the corresponding values of the averaging error estimator.
    For comparison, we show convergence of the $L^2$-error (dash-dotted with x), where the grid adaptation is based on the $L^2$-error itself. The dotted line with o shows the values of the corresponding averaging error estimator on that grid.\label{fig:AverageConvergencerates}}
\end{figure}
\section{Conclusions and discussion}\label{sec:conclusion}
We developed and analyzed a discontinuous Galerkin approximation for the radiative transfer equation in slab geometry. The use of quad-tree grids allowed for a relatively simple analysis with similar arguments as for more standard elliptic problems. While such grids allow for local mesh refinement in phase-space, the implementation of the numerical scheme is straightforward. For sufficiently regular solutions, we showed optimal rates of convergence. 

We showed by example that non-smooth solutions can be approximated well by adaptively refined grids.
The ability to easily adapt the computational mesh can also be useful when complicated geometries must be resolved, which may occur in higher-dimensional situations.
Also more general elements could be employed at the expense of a more complicated notation and analysis; we leave this to future research.
In order to automate the mesh adaptation procedure, an error estimator is required. 
We investigated numerically hierarchical error estimators and estimators based on local averaging in a post-processing step. All three estimators closely follow the actual error, and, in the case of point singularities, they can be used to obtain optimal convergence rates. We note that the hierarchical error estimators require to solve global problems, and it is left for future research to investigate whether a localization is possible. 
Upper bounds for the error can be derived for consistent approximations using duality theory \cite{Han_2014}. Rigorous a posteriori error estimation has also been done using discontinuous Petrov-Galerkin discretizations \cite{dahmen2020adaptive}. We leave it to future research to analyze the error estimators for the discontinuous Galerkin scheme considered here and to generalize the present method to a corresponding $h-p$ version, where the polynomial degrees can be varied independently over the elements.

While the solution of the linear systems for uniformly refined grids can be implemented using the established preconditioned iterative solvers \cite{AdamsLarsen02,palii2020convergent}, the structure of the linear systems for adaptively refined grids is more complex because the equations do not fully decouple in $\mu$; compare the situations in \Cref{fig:quadmesh}. One possible direction is to develop nested solvers, or to adapt the methodology of \cite{Pazner2022}. We leave this for future research.

Another direction of future research entails the regularity of the right-hand side $f$ in \cref{eq:ep1} and $g$ in \cref{eq:ep2}. If $f$ and $g$ define only an element in the dual space of $V$, see \cref{eq:ep_weak}, then the flux $\sigma_t^{-1}\dz u\notin V$ in general, and thus the flux may not have a trace. In this low regularity regime, the analysis of \Cref{sec:Sec4_numerical_examples_PS_DG} cannot be carried out. A possible remedy might be to use a lifting operator to replace the face integral by integrals over $\Omega$, see, e.g., \cite[p.~138]{diPietroErn} or \cite{Liu2022}.

\section*{Acknowledgements}
RB and MS acknowledge support by the Dutch Research Council (NWO) via grant OCENW.KLEIN.183.
http://dx.doi.org/10.13039/501100003246, "Nederlandse Organisatie voor
Wetenschappelijk Onderzoek". 

\bibliographystyle{plain}
\bibliography{sources_new}
\end{document}